\def\bT{\bar{T}}
\def\bP{\bar{P}}

\def\Lim{\mathop{\rm Lim}}
\def\ult{\mbox{Ult}}
\def\rR{\mbox{R}}
\def\rP{{\mathop{\rm P}}}
\def\bj{\bm{J}}
\def\bJ{\bm{J}}
\def\bM{\bm{M}}
\def\old{{\mathop{o}}}
\def\Tr{{\mathop{\rm Tr}}}
\def\tr{\mathop{{\rm Tr}}}
\def\rud{\mbox{rud}}

\def\bDiamond{{\mbox{\large$\Diamond$}}}

\def\bsot{{\mathbf \Sigma}^1_2}

\def\aaqd{{\aaq}_{\delta}}
\def\vs{\vec{s}}

\def\zfc{\mbox{{ZFC}}}
\def\CAA{C(\hspace{-1pt}\aaq\hspace{-1pt})}
\def\aaq{\mathop{{\mbox{\tt a\hspace{-0.4pt}a}}}}
\def\aaqs{\mathop{\mbox{\tt a\hspace{-0.4pt}a\hspace{0.6pt}$s$\hspace{0.6pt}}}}
\def\aaqu{\mathop{\mbox{\tt a\hspace{-0.4pt}a\hspace{0.6pt}$u$\hspace{0.6pt}}}}
\def\aaqt{\mbox{\tt a\hspace{-0.4pt}a\hspace{0.6pt}$t$\hspace{0.6pt}}}
\def\aaqss{\mbox{\tt a\hspace{-0.4pt}a\hspace{0.6pt}$\bs$\hspace{0.6pt}}}
\def\aal{\mbox{\tt aa}^{\hspace{-0.0pt}{0}}}
\def\aaqt{\mbox{\tt a\hspace{-0.4pt}a\hspace{0.6pt}$t$\hspace{0.6pt}}}
\def\aam{\mbox{\tt aa}^{\hspace{-0.7pt}-}}
\def\stat{\mbox{\tt s\hspace{-1pt}t\hspace{-1pt}a\hspace{-1pt}t}\hspace{1pt}}

\def\PFA{\mbox{\rm PFA}}

scaled \magstephalf 

\def\restriction{\upharpoonright}
\def\rest{\upharpoonright}

\def\cof{\mathop{\rm cf}}

\def\rng{\mathop{\rm rng}}

\def\dom{\mathop{\rm dom}}
\def\ran{\rng}

\def\bx{\vec{x}}
\def\bs{\vec{s}}
\def\bt{\vec{t}}

\def\bz{\vec{z}}
\def\by{\vec{y}}
\def\ba{\vec{a}}
\def\bb{\vec{b}}
\def\bc{\vec{c}}

\def\Pw{{\cal P}}
\def\On{{\rm On}}
\def\la{\langle}
\def\ra{\rangle}

\def\ma{{\mathcal A}}

\def\mm{{\mathcal M}}
\def\mn{{\mathcal N}}

\newcommand{\DD}{{\cal D}}

\newcommand{\I}{{\cal I}}
\newcommand{\J}{{\cal J}}
\newcommand{\F}{{\cal F}}

\newcommand{\cf}{\mathop{\rm cf}}

\mathchardef\bfSigma="0606 \mathchardef\bfPi="0605
\mathchardef\bfDelta="0601

\def\rest{\mathord{\restriction}} 
\newcommand{\force}{\Vdash} 

\def\rest{\mathord{\restriction}}
\newcommand{\open}{\Bbb}

\newcommand{\oP}{{\open P}}
\newcommand{\oQ}{{\open Q}}
\newcommand{\oR}{{\open R}}

\def\DD {{\cal D}}

\makeatletter

\def\bracketdown#1{\mathop{\vbox{\ialign{##\crcr\noalign{\kern2\p@}
\downbracketfill\crcr\noalign{\kern2\p@\nointerlineskip}
$\hfil\displaystyle{#1}\hfil$\crcr}}}\limits}

\def\bracketup#1{\mathop{\vbox{\ialign{##\crcr\noalign{\kern1\p@}
\upbracketfill\crcr\noalign{\kern1\p@\nointerlineskip}
$\hfil\displaystyle{#1}\hfil$\crcr}}}\limits}

\def\upbracketfill{$\m@th
\makesm@sh{\llap{\vrule\@height2\p@\@width.4\p@}}%
\leaders\vrule\@height.4\p@\hfill
\makesm@sh{\rlap{\vrule\@height2\p@\@width.4\p@}}$}

\def\downbracketfill{$\m@th
\makesm@sh{\llap{\vrule\@height.4\p@\@depth1.6\p@\@width.4\p@}}%
\leaders\vrule\@height.4\p@\hfill
\makesm@sh{\rlap{\vrule\@height.4\p@\@depth1.6\p@\@width.4\p@}}$}
\makeatother

\def\CH{\text{CH}}
\def\GCH{\text{GCH}}

\def\\(\two\){Eloise}

\def\a{\alpha}

\def\g{\gamma}

\def\bab{\bar{\beta}}

\def\two{\mbox{$\mathbf{II}$}}

\def\DEF{\mbox{Def}}
\def\phi{\varphi}
\def\hod{\mbox{\rm HOD}}\def\HOD{\mbox{\rm HOD}}

\documentclass[12pt]{article} 
\usepackage{amsthm} 
\usepackage{graphicx} 
\usepackage{amssymb,latexsym,times,amsmath,bm}
\newcommand{\psfrag}[2]{}
\usepackage{makeidx,proof,color}
\usepackage{verbatim,hyperref,lineno,comment}
\usepackage[utf8]{inputenc}
\DeclareMathAlphabet{\mathpzc}{OT1}{pzc}{m}{it}
 \theoremstyle{plain}
  \newtheorem{theorem}{Theorem}[section]
  \newtheorem{lemma}[theorem]{Lemma}
  \newtheorem{corollary}[theorem]{Corollary}
  \newtheorem{proposition}[theorem]{Proposition}

  \theoremstyle{definition}
  \newtheorem{definition}[theorem]{Definition}
  \newtheorem{example}[theorem]{Example}

  \theoremstyle{remark}

 \newtheorem*{claim}{Claim}

\def\LL{\mathcal{L}}
\newcounter{footnoteValueSaver}

\begin{document}

\author{Juliette Kennedy\thanks{Research partially supported by grant 40734 of the Academy of Finland.}\\ Helsinki \and Menachem Magidor\thanks{Research supported by the Simons Foundation and  the Israel Science Foundation grant 684/17.}\\ Jerusalem \and Jouko V\"a\"an\"anen\thanks{Research supported by the Simons Foundation, the Faculty of Science of the University of Helsinki,  and grant 322795 of the Academy of Finland.}\\ Helsinki and Amsterdam}

\title{Inner Models from Extended Logics: 
Part 2\thanks{The authors are grateful to the American Institute of Mathematics for support and to Gabriel Goldberg,  Paul Larson, Otto Rajala, Ralf Schindler, John Steel,  Philip Welch, Trevor Wilson, Hugh Woodin and Ur Yaar for comments on the results presented here. This project has received funding from the European Research Council (ERC) under the
European Union’s Horizon 2020 research and innovation programme (grant agreement No
101020762).}
}
  
\maketitle

\begin{abstract}

We introduce a new inner model $\CAA$ arising from stationary logic. We show that assuming a proper class of    Woodin cardinals, or alternatively $\PFA$, the regular uncountable
cardinals of $V$ are measurable in the inner model $\CAA$ and $\CAA$ satisfies $\CH$. Moreover, assuming a proper class of    Woodin cardinals, the theory of $\CAA$ is (set) forcing absolute. We introduce an auxiliary concept that we call Club Determinacy, which simplifies the construction of $\CAA$ greatly but may have also independent interest.  Based on Club Determinacy, we  introduce the concept of aa-mouse which we use  to prove $\CH$ and other properties of the inner model $\CAA$. 
\end{abstract}


\section{Introduction}

This is the second part of a two-part paper on inner models obtained by means of extended logics. 
A generally acknowledged weakness of G\"odel's in many ways robust inner model $L$ is that it cannot support large cardinals, beyond such ``small" large cardinals as  inaccessible, Mahlo, and weakly compact cardinals. In the so-called \emph{Inner Model Program} inner models are built for bigger and bigger large cardinals, reaching currently as far as a Woodin limit of Woodin cardinals. These models resemble G\"odel's $L$ in that deep \emph{fine-structure} can be established for them leading, among other things, to canonical proofs of $\CH$, $\bDiamond$, $\Box$, etc. in those inner models. While these so-called fine-structural inner models are extremely useful in almost all areas of modern set theory, it cannot be denied that they are built somewhat ``opportunistically", by assuming a large cardinal and building a carefully crafted model around it. With our new inner models we look for a more canonical inner model construction which would still have desirable properties. 

But what should one expect from a canonical inner model? First of all we propose that we should expect \emph{robustness}. We have in mind three meanings of robustness: (1) Stability of the model under changes in the definition (in the
fixed universe of set theory). 
(2) Robustness across universes of set theory, stability under
forcing extensions.
(3) The theory of the model (or an important part of it) should be
invariant under forcing extensions.
A second quality we propose to expect from a canonical inner model is \emph{completeness} in the sense that  canonical definable objects should be included.
A litmus test of this would be  closure under sharps or other canonical operations.

The first part  \cite{kmv} of this two-part paper dealt mainly with some general questions concerning  inner models obtained from extended logics, and more specifically the inner model $C^*$ defined by means of the cofinality quantifier  \cite{MR0376334}. In this second part we focus on the a priori bigger inner model $\CAA$ defined by means of the stationary logic \cite{MR486629}. Note that
\begin{equation}\label{relations}
L\subseteq C^*\subseteq \CAA\subseteq \hod.
\end{equation}

The main results about $C^*$ in \cite{kmv} were that under the assumption of a proper class of Woodin cardinals, the theory of $C^*$ is set forcing absolute, uncountable cardinals $>\omega_1$ of $V$ are weakly compact in $C^*$ (and $\omega_1$ is Mahlo), and the theory of $C^*$ is independent of the cofinality used. Moreover, $C^*$ is closed under sharps. We were not able to solve the problem of $\CH$ in $C^*$ although we showed, assuming three Woodin cardinals and a measurable above them, that for a cone of reals $r$ the relativized inner model $C^*[r]$  satisfies $\CH$.

Here
we show that if there is a proper class of Woodin cardinals, then uncountable cardinals of $V$ are measurable in $\CAA$,  and the theory of the model $\CAA$ is invariant under set forcing. This raises naturally the question of the truth-value of $\CH$ in $\CAA$. We show, assuming a proper class of Woodin cardinals, or alternatively $\PFA$, that $\CAA$ satisfies $\CH$. Again, we point out that $\CAA$ is closed under sharps. We also consider some variants of $\CAA$.

The models $C^*$ and $\CAA$ arise from general considerations involving such basic set-theoretical concepts as cofinality and stationarity. It is quite remarkable that we can achieve the level of robustness that these models manifest. It should come as no news that we have to make set-theoretical assumptions before we can obtain robustness results for $C^*$ and $\CAA$. For example, if $V=L$, then both models are simply identical to $L$. Our assumptions are either large cardinal axioms or forcing axioms. 

There are two new tools that we develop for the proofs of the results mentioned. The first tool is \emph{Club Determinacy} which simplifies stationary logic considerably in our construction. Roughly speaking, Club Determinacy says that every stationary definable set of countable subsets of $\CAA$ contains a club. The second tool is the concept of an \emph{aa-mouse}. Roughly speaking, an aa-mouse consists of a transitive set together with a theory formulated in stationary logic. Intuitively speaking, the transitive set satisfies the theory-part, but this is not true in general. For example, it is not true  if the transitive set is countable. The major part of this paper is devoted to proving Club Determinacy under large cardinal assumptions, or $\PFA$, and to developing the theory of aa-mice and, what we call, aa-ultrapowers of aa-mice.

We feel that there are a wealth of questions worth studying about the new inner models. At the end of the paper we list some such questions.
\medskip

\noindent{\bf Notation:} If $\kappa$ is a cardinal and $M$ a set, we denote the set of subsets  $M$ of cardinality $<\kappa$ by  $\Pw_{\kappa}(M)$. We use vector notation $\ba,\bb,\bx$ etc for finite sequences.
$\forall\bx\phi$ is short for $\forall x_1\ldots\forall x_n\phi$ and 
$\aaq\bs\phi$ is short for $\aaqs_1\ldots\aaqs_n\phi$. If $h$ is a function and $x\subseteq\dom(f)$, then we use $h[x]$ to denote the set $\{h(y) : y\in x\}$.
$H(\mu)$ is the set of sets of hereditary cardinality less than $\mu$. The class of limit ordinals is denoted $\Lim$.

\section{Basic concepts}

Let us recall that a set $S$ of countable subsets of a set $M$ is said to be \emph{closed unbounded} (club) if for every countable $s\subseteq M$ there is $s'\in S$ such that $s\subseteq s'$, and for every 
$\{s_n:n<\omega\}\subseteq S$ such that $\forall n(s_n\subseteq s_{n+1})$ the set $\bigcup_ns_n$ is in $S$, or equivalently, $S$ is the set of countable subsets of $M$ closed under a fixed given countable set of functions. The set $S$ is \emph{stationary} if it meets every club set of countable subsets of $M$. \emph{Stationary logic}  is the extension of first order logic by the following second order quantifier:

\begin{definition}If $\ba$ is a finite sequence of elements of $M$ and $\bt$ is a finite sequence of countable subsets of $M$, then we define
$$\mm\models\aaqs\phi(s,\bt,\ba)$$ if and only if $\{A\in \Pw_{\omega_1}(M)\ :\ (\mm,A)\models\phi(A,\bt,\ba)\}$ contains a club of countable subsets of $M$.  We denote $\neg\aaqs\neg\phi$ by $\stat s\phi$. The extension of first order logic by the quantifier $\aaq$ is denoted $\LL(\aaq)$.
\end{definition}

This quantifier was essentially introduced in \cite{MR0376334} and studied extensively in \cite{MR486629}. 
The idea is that rather than asking whether there is \emph{some} countable set $A$ satisfying $\phi(A)$, or whether \emph{all}  countable sets $A$ satisfy $\phi(A)$, we ask whether most  $A$ satisfy $\phi(A)$. The second order ``some/all" quantifiers are generally believed to be too strong to give rise to interesting model theory, but the ``most" quantifier has turned out to be better behaved. There is a complete axiomatization,  a Compactness Theorem in countable vocabularies, and a Downward L\"owenheim-Skolem Theorem down to $\aleph_1$ for countable theories (i.e. every countable consistent theory has a model of cardinality $\aleph_1$).

Some examples of the expressive power of stationary logic are the following:
We can express
 ``$\phi(\cdot)$ is countable''  with $\aaqs\forall y (\phi(y)\to s(y))$.
 If we have a linear order $\phi(\cdot,\cdot)$, we can express  it having cofinality $\omega$
 with $\aaqs\forall x \exists y(\phi(x,y)\wedge s(y))$. We can express   $\phi(\cdot,\cdot)$ being  $\aleph_1$-like
with $\forall x\aaqs\forall y (\phi(y,x)\to s(y))$. 
The set $\{\alpha<\kappa\ :\ \cof(\alpha)=\omega\}$ is $\LL(\aaq)$-definable on $(\kappa,<)$ by means of $\aaqs(\sup(s)=\alpha)$, where $\sup(s)$ is a shorthand notation for the supremum of $s$.  The property of a set $A\subseteq \{\alpha<\kappa\ :\ \cof(\alpha)=\omega\}$  being stationary is definable in $\LL(\aaq)$ by means of  $\stat s(\sup(s)\in A)$.
Finally, we can express an $\aleph_1$-like linear order  $\phi(\cdot,\cdot)$  containing a closed unbounded subset (i.e. a copy of $\omega_1$)
with  $\aaqs(\sup(s)\in\dom(\phi))$.

We will need below the concept of \emph{relativisation} of $\LL(\aaq)$-formulas. 
Relativisation is defined inductively as in first order logic except  that  the relativisation $(\aaqs\psi(s))^{(x)}$ of $\aaqs\psi(s)$ to $x$ is defined as $\aaqs((\psi(s\cap x))^{(x)})$,
where $\psi(s\cap x)$ denotes the formula obtained from $\psi(s)$ by replacing everywhere $y\in s$ by $y\in s\wedge y\in x$.

The  {\em axioms} of the logic $\LL(\aaq)$ are \cite{MR486629}:

\begin{equation}\label{axioms}
\left.\begin{tabular}{l}
 $(A0)\quad \aaqs\phi(s)\leftrightarrow\aaqt\phi(t)$\\
 $(A1)\quad \neg\aaqs(\bot)$\\
 $(A2)\quad \aaqs(x\in s)$, $\aaqt(s\subseteq t)$\\
 $(A3)\quad (\aaqs\phi\wedge\aaqs\psi)\to\aaqs(\phi\wedge\psi)$\\
 $(A4)\quad \aaqs(\phi\to\psi)\to(\aaqs\phi\to\aaqs\psi)$\\
 $(A5)\quad \forall x\aaqs\phi(x,s)\to\aaqs\forall x\in s\phi(x,s)$. 
\end{tabular}\right\}
\end{equation}
 
\noindent The rules are Modus Ponens, the usual rule of generalisation and the new rule of $\aaq$-generalisation i.e. if $T\vdash\phi\to\psi$ and $s$ is not free in $T\cup\{\phi\}$, then $T\vdash\phi\to\aaqs\psi$. These are complete in the sense that any countable $\LL(\aaq)$-theory consistent with them has a model of cardinality $\aleph_1$. Intuitively, (A1) says that $\emptyset$ is not club. (A2) says that the set of countable sets having a fixed element as an element, as well as the set of countable sets containing a fixed countable set as a subset, are club. (A3) and (A4) simply say that the club-filter (of definable sets) is a filter. Finally, (A5) is a formulation of Fodor's Lemma. 
 
 Suppose $A$ is a stationary subset  of a regular $\kappa>\omega$ such that  $\forall\alpha\in A(\cof(\alpha)=\omega)$.
 The {$\omega$-club filter} $\F^\omega(A)$  is the set of subsets of $A$ which contain the intersection of $A$ with a club  subset of $\kappa$.
Note that $\F^\omega(A)$ is $<\!\kappa$-closed.
   The property of $B\subseteq\kappa$  belonging to $\F^\omega(A)$ is definable from $A$ in $\LL(\aaq)$ by means of $\aaqs(\sup(s)\in A\to \sup(s)\in B)$.


There are generalizations of the notions of club
and stationarity from $\Pw_{\omega_1}(A)$ to $\Pw_{\lambda} (A)$, where $\lambda$ is a regular cardinal.
 Since there are slight variations in the way clubs and stationary
subsets of $\Pw_{\lambda}(A)$ are defined, we specify below what we mean by this
terminology.

\begin{definition}
$C\subseteq \Pw_{\lambda}(A)$ is \emph{closed unbounded} (club) in $\Pw_{\lambda}(A)$
if for every $X \in \Pw_{\lambda}(A)$ there is $Y \in C$ such that $X\subseteq Y$ and, moreover,
if $\la X_j:j<\delta\ra$, $\delta<\lambda$, is an increasing sequence of members of $C$,  then
$\bigcup_{\alpha<\delta}X_\alpha$ is in $C$. A set $S\subseteq\Pw_\lambda(A)$ is called \emph{stationary} in $\Pw_\lambda(A)$ if it meets every club of $\Pw_\lambda(A)$.

\end{definition} 

If $\lambda\subseteq A$, then $\{X\in \Pw_\lambda(A):X\cap\lambda\in\lambda\}$ is a club in $\Pw_\lambda(A)$.
Also, if $\lambda\subseteq A$ then $D\subseteq\Pw_\lambda(A)$ contains a club if and only if there
 is an algebra on $A$ (with countably many operations) such that (the domains of) all subalgebras whose intersection with $\lambda$ is an ordinal, are in $D$.

If $\delta$ is an uncountable cardinal such that $\delta=\delta^{<\delta}$, we consider the quantifier ${\aaq}_{\delta}$ with the following meaning: If $\ba$ is a finite sequence of elements of $M$ and $\bt$ is a finite sequence of  subsets of $M$ of cardinality $<\delta$, then we define
$$\mm\models{\aaq}_{\delta} s\phi(s,\bt,\ba)$$ 
 if and only if $\{A\in \Pw_{\delta}(M)\ :\ (\mm,A)\models\phi(A,\bt,\ba)\}$ contains a club of  $\Pw_{\delta}(M)$. It is proved in  \cite{MR819644} that
a sentence of $\LL(\aaq)$ has a model if and only if it has a model when $\aaq$ is interpreted as ${\aaq}_{\delta}$.
\medskip

\section{Inner model $\CAA$}

The idea is that $\CAA$ is the inner model that results if in the usual definition of G\"odel's constructible hierarchy $L$ the role of first order logic as a vehicle of definability is played by stationary logic. In fact, we model our definition of $\CAA$ more in the style of Jensen's $J$-hierarchy \cite{MR309729}, which is, after all, equivalent to G\"odel's $L$-hierarchy. We add stationary logic to the usual definition of the $J$-hierarchy. The addition takes place by adding the truth-definition of stationary logic as a special predicate to the usual definition.

Since the definition of $\CAA$ applies to any logic $\LL^*$, we formulate the following definition for an arbitrary logic $\LL^*$: 

\begin{definition}\label{definJ}
Suppose $\LL^*$ is a logic the sentences of which are\footnote{For the sake of simplicity.} (coded by) natural numbers.  We define the hierarchy $(J'_\alpha)$, $\alpha\in\Lim$, of {\em sets constructible using} $\LL^*$ %
and the class $\Tr$,  by  transfinite double induction,  as follows:
\footnote{The vocabulary of $\phi(\bx)$, $\bx=(x_1,\ldots, x_n)$,  below consists of two binary predicate symbols. The sentence
$\phi(\ba)$, $\ba=(a_1,\ldots, a_n)$, is the result of substituting a constant symbol $c_{a_i}$, denoting $a_i$, for $x_i$ for $i=1,\ldots, n$. We generally use $a_i$ to denote (also) the constant symbol $c_{a_i}$ when no confusion arises.}
$$\Tr=\{(\alpha,\phi(\ba)): (J'_\alpha,\in,\Tr\rest \alpha)\models\phi(\ba), \phi(\bx)\in\LL^*,\ba\in J'_\alpha, \alpha\in\Lim\},$$
where $$\Tr\rest \alpha=\{(\beta, \psi(\ba))\in \Tr:\beta\in \alpha\cap\Lim\},$$
and
\def\arraystretch{1.3}
\begin{equation}\label{J-hier}
\left\{\begin{array}{lcl}
J'_{0}&=&\emptyset\\
J'_{\alpha+\omega}&=&\rud_{\Tr}(J'_\alpha\cup\{J'_\alpha\})\\
J'_{\omega\nu}&=&\bigcup_{\alpha<\nu}J'_{\omega\alpha},\mbox{ for }\nu\in\Lim.
\end{array}\right.
\end{equation}
\def\arraystretch{1}
Here the rudimentary closure operation $\rud_{\Tr}$ includes the operation $x\mapsto x\cap \Tr$.
 We use $C(\LL^*)$ to denote the class $\bigcup_{\alpha\in\Lim}J'_\alpha$ and use $\bj'_\alpha$ to denote the structure $(J'_\alpha,\in,\Tr\rest\alpha)$. 
\end{definition}

Additionally, we denote $$\Tr_\alpha=\{\phi(\ba):(\alpha,\phi(\ba))\in\Tr\},$$ whence
$$\Tr=\bigcup\{\{\alpha\}\times \Tr_\alpha: {\alpha\in\Lim}\}.$$

The point of the definition of $\bj'_\alpha$ is that we do not only add in successor stages sets that are definable (or  images under rudimentary functions) from elements of the lower levels but we also add a truth-definition $\Tr\rest \alpha$ which makes reference to definable sets particularly smooth. In particular, it helps us produce a uniformly definable well-order of each of the levels $J'_\alpha$.

In the special case that $\LL^*$ is $\LL(\aaq)$, we denote $C(\LL(\aaq))$ by $$\CAA.$$
We also consider the inner model $C({\aaq}_{\delta})$ i.e $C(\LL({\aaq}_{\delta}))$. Since the quantifier $Q^{\cf}_\omega$, which gives rise to the inner model $C^*$ ($=C(\LL(Q^{\cf}_\omega))$) is definable in $\LL(\aaq)$, we have the trivial relations  of (\ref{relations}).

\begin{lemma}\label{well-order}
$\CAA$ is a model of $ZFC$. The model $\CAA$ has a canonical (first order) definable well-order $\prec$.
\end{lemma}

\begin{proof}
The claim follows from  general properties of the $J$-hierarchy (see e.g. \cite[Lemma 5.26]{MR3243739}).
\end{proof}

We recall the following connection between the $J$-hierarchy of Jensen and the  $L$-hierarchy of G\"odel\footnote{Note that we do not claim that the structures $J'_{\alpha}$ are amenable.} in the definition of $\CAA$:

\begin{lemma}[\cite{MR309729,MR3243739}] \label{defty} Suppose $(J'_\alpha)$ is the hierarchy generating $\CAA. $ A set $A\subseteq J'_\alpha$ is in 
$J'_{\alpha+\omega}$ if and only if   there are a first order formula $\phi(x,y)$ and $b\in J'_\alpha$ such that $A=\{a\in J'_\alpha : (J'_{\alpha},\in,\Tr\rest{\alpha}, \Tr_{\alpha})\models\phi(a,b)\}$.
\end{lemma}

We used a different definition for $C(\LL^*)$ in  \cite{kmv}. There we introduced an inner model obtained in the same way as G\"odel's constructible hierarchy $L$, but replacing in the definition first order logic by the logic $\LL^*$. The general construction was as follows:

\begin{definition}[\cite{kmv}]\label{defin}
Suppose $\LL^*$ is a logic. If $M$ is a set, let $\DEF_{\LL^*}(M)$ denote the set of all sets of the form $X=\{a\in M : (M,\in)\models\phi(a,{\bb})\},$ where $\phi(x,{\by})$ is an arbitrary formula of the logic $\LL^*$ and ${\bb}\in M$. We define the hierarchy $(L'_\alpha)$ as follows:

\begin{center}
\begin{tabular}{lcl}
$L'_0$&=&$\emptyset$\\
$L'_{\alpha+1}$&$=$&$\DEF_{{\mathcal L}^*}(L'_\alpha)$\\
$L'_\nu$&$=$&$\bigcup_{\alpha<\nu}L'_\alpha\mbox{ for limit $\nu$}$.\\
\end{tabular} 
\end{center}
\end{definition}

Let us use $C_{\old}(\LL^*)$ to denote the class $\bigcup_\alpha L'_\alpha$.
In the special case that $\LL^*$ is $\LL(\aaq)$, we denote $C_{\old}(\LL(\aaq))$ by $C_{\old}(\aaq).$
The reason for changing the definition from $C_{\old}(\LL^*)$ (as in  \cite{kmv}) to the current $C(\LL^*)$ is that it turned out to be unclear, as pointed out by Gabriel Goldberg\footnote{Personal communication.}, whether the former satisfies the Axiom of Choice. For  the logic $\LL(Q^{\cf}_\omega)$ there is no difference: 
$C(\LL(Q^{\cf}_\omega))=C_{\old}(\LL(Q^{\cf}_\omega)) (=C^*)$. We could give the new definition of $C(\LL^*)$ in terms of the $L$-hierarchy instead of the $J$-hierarchy, because of the close relationship between the two hierarchies, see e.g. \cite[\S 2.4]{MR309729}, but the $J$-hierarchy is more convenient because its levels are closed under the pairing function which we need to code finite sequences used in the definition of ${\tr}_\alpha$. 

 In the course of this paper we will see that $\CAA$ is in many ways a fairly robust inner model in the sense of our Introduction, at least if there are big enough large cardinals. 


It is important to keep in mind that the quantifier $\aaqs$ in the construction of $\CAA$ asks whether there is a club in $V$ of countable sets $s$ in $V$ with some property. Neither the club nor the countable sets need be in $\CAA$. Thus, although we focus on an inner model $\CAA$, we let the quantifier  $\aaq$ ``reach out" to $V$. Thus $\CAA$ knows certain facts about $V$ but it may not be able to have witnesses to corroborate those facts. The whole point of using $\LL(\aaq)$ in the definition of $\CAA$ is that $\LL(\aaq)$ provides \emph{some} information about $V$ but not {too} {much}.  

The countable levels $J'_\alpha$, $\alpha<\omega^V_1$, bring nothing new, although $\omega_1^V$ may be a large cardinal in $\CAA$. They are the same as the respective levels of the constructible hierarchy, as the $\aaq$-quantifier is eliminable in countable models.

Note that $S=\{\alpha<\kappa\ :\ \cof^V(\alpha)=\omega\}\in \CAA$.
 The property of $A\subseteq S$ of being stationary (in $V$) is definable in $\CAA$, as is the property of 
 containing the $\omega$-cofinal elements of a club.
 Thus, if $A\in \CAA$, then the ``trace" of the $\omega$-club filter of $V$ on $A$, namely 
 $(\F^\omega(A))^V\cap \CAA$, is in $\CAA$. One of the main results of this paper is that $(\F^\omega(\kappa))^V\cap \CAA$ is a normal ultrafilter on $\kappa$, whenever $\kappa>\omega$ is regular, assuming large cardinals. 
 
The following robustness property of $\CAA$ is often useful: 
 
\begin{proposition}\label{sigmaclosed}
Suppose $\oP$ is a $\sigma$-closed notion of forcing and $G$ is $\oP$-generic. Then $\CAA^V=\CAA^{V[G]}$.
\end{proposition}

\begin{proof} Let $\bJ''_\alpha$ be the $\bJ'_\alpha$ as computed in $V[G]$.
We prove $\bJ'_\alpha=\bJ''_\alpha$ for all $\alpha$, by induction on $\alpha$. The induction step boils down to  the following claim: If $S\subseteq \Pw_{\omega}(J'_\alpha)$ is an $\LL(\aaq)$-definable set with parameters in $V$ and $S\in V$, then $S$ is stationary in $V$ if and only if $S$ is stationary in $V[G]$. To prove this, let us assume $S$ is stationary in $V$. Then $S$ is stationary in $V[G]$ because $\oP$ is proper. On the other hand, if $S\in V$ is stationary in $V[G]$, then $S$ is obviously stationary in $V$.
\end{proof}

Many natural questions about $\CAA$ immediately suggest themselves:
\begin{itemize}
\item  Does it satisfy $\CH$? 
\item Does it have large cardinals? 
\item How absolute is it? 
\item Is its theory forcing absolute? 
\item How is it related to other known inner models such as $L$, $\HOD$, etc?
\end{itemize}

We will provide some answers in this paper, but many natural questions remain also unanswered.
We shall prove $\CAA\models \CH$ from large cardinal assumptions, but let us immediately observe that ZFC alone does not limit the cardinality of the continuum in $\CAA$ to either $\aleph_1$ or to $\le\aleph_2$. This is in sharp contrast to the case of $C^*$ (see \cite{kmv}) where the continuum is always at most $\aleph_2$ of $V$. 

\begin{theorem}\label{slllk}
$Con(ZF)$ implies  $Con(|\oR\cap \CAA|\ge\aleph_3^V)$. 
\end{theorem}

\begin{proof}
Assume  $V=L$. Let $S\subseteq\omega_3$ be a non-reflecting stationary set of ordinals of cofinality $\omega$ with fat  complement (i.e. for every club $C\subseteq\omega_3$, $C\setminus S$ contains closed sets of ordinals of arbitrarily large order types below $\omega_3$). Let $S_\alpha$, $\alpha<\omega_3$, be a partitioning of $S$ into disjoint stationary sets. Let us now work in a generic extension obtained by adding Cohen reals $r_\alpha$, $\alpha<\omega_3$. The sets $S_\alpha$ are still stationary, because the forcing is CCC. Let $A$ be the set $\{\omega\cdot\alpha+n : n\in r_\alpha, \alpha<\omega_3\}$. Let $E$ be the union of the sets $S_\alpha$, where $\alpha\in A$. Let us move to a forcing extension obtained by forcing a club $D$ through the fat stationary set $\omega_3\setminus E$. This forcing does not add bounded subsets of $\omega_3$, whence $\omega_3$ does not change. If $\alpha\in A$, then $S_\alpha\cap D=\emptyset$ and $S_\alpha$ is therefore non-stationary after the forcing. On the other hand, If $\alpha\notin A$, then $S_\alpha\cap E=\emptyset$ and shooting a club through $\omega_3\setminus E$ preserves the stationarity of $S_\alpha$.  Hence for any $\alpha\in\omega_3$, $\alpha\in A$ if and only if $S_\alpha$ is non-stationary. Hence $A\in \CAA$.
Now $r_\alpha=\{n : \omega\cdot\alpha+n\in A\}$. Hence each $r_\alpha$ is in $\CAA$, and therefore $|\oR\cap \CAA|\ge\aleph_3^V$.
\end{proof}

The role of $\aleph_3$ in the above theorem is not crucial, but just an example. It can be replaced by any cardinal. Since $ZFC\vdash|\oR\cap C^*|\le\aleph_2$ \cite{kmv}, we obtain\footnote{Work in progress by a SQuaRE group shows that $C^*\ne \CAA$ follows also from the existence of a measurable cardinal of Mitchell-order $>1$.} 
:

\begin{corollary}
$Con(ZF)$ implies $Con(C^*\ne \CAA)$.
\end{corollary}


We can use the proof of Theorem~\ref{slllk} to prove the consistency of the non-absoluteness of $\CAA$ in the sense that inside $\CAA$ the $\CAA$ may look different than from outside:

\begin{proposition}
$Con(ZF)$ implies $Con(\CAA^{\CAA}\ne \CAA)$\footnote{Ur Ya'ar  has proved stronger results, see \cite{https://doi.org/10.48550/arxiv.2209.10247}.}.
\end{proposition}

\begin{proof}We proceed as in the proof of Theorem~\ref{slllk}.
Assume $V=L$.  Let $S_n$, $n<\omega$, be a partitioning of $\omega_1$ into disjoint stationary sets. Let us then work in a generic extension obtained by adding a Cohen real $r$. The sets $S_\alpha$ are still stationary.  Let $E$ be the union of $S_n$, where $n\in r$. Let us move to a forcing extension $V[G]$ obtained by forcing a club $D$ through the stationary set $\omega_1\setminus E$. Now $n\in r$ if and only if $S_n$ is non-stationary. Hence $r\in \CAA$, whence $L(r)\subseteq \CAA$ and $\CAA\ne L$.
One can prove by induction on $\alpha$ that $(J'_\alpha)^{V[G]}=(J'_\alpha)^{L[r]}$. The non-trivial step says that if $T\subseteq P_{\omega_1}(B)$ where $B,T\in L[r]$, then $T$ is stationary in $V[G]$ iff for some $n\in\omega\setminus\{r\}$ the set $\{s\in T|s\cap\omega_1\in S_n\}$ is a stationary subset of $P_{\omega_1}(B)$ in $L[r]$.
      Note that  all countable subsets of  $V[G]$ are in $L[r]$. Hence $\CAA\subseteq L(r)$, 
and, in consequence, we obtain $\CAA=L(r)$. But $\CAA^{L(r)}=L$. Hence $V[G]\models \CAA^{\CAA}\ne \CAA$.
\end{proof}

If $x\in \CAA$ and $x^\#$ exists, then $x^\#\in \CAA$. This is proved as for $C^*$ in \cite{kmv}. If $L^\mu$ exists, then $L^\nu\subseteq C^*$ for some $\nu$, and hence $L^\nu\subseteq \CAA$  for some $\nu$. However, we do not know whether  $L^\mu$, where $\mu$ is a measure on the smallest possible ordinal,  is contained in $\CAA$.

\section{Club determinacy}

We  introduce the useful auxiliary concept of Club Determinacy and show that $\CAA$ satisfies it, assuming large cardinals or $\PFA$. Roughly speaking, Club Determinacy says that definable sets of ordinals of cofinality $\omega$ in $\CAA$ either contain a club or their complement contains a club. This simplifies the structure of $\CAA$ as  we do not have any definable stationary co-stationary sets. The main results of the later sections are heavily based on this.


\begin{definition}[\cite{MR556893}]\label{cd}
A first order structure $\mm$ is {\em club determined\footnote{In \cite{MR556893} the name ``finitely determinate" is used.}}
  if \begin{equation}\label{cd1}
\mm\models
\forall\bx[\aaqs\phi(\bx,s,\bt)\vee
\aaqs\neg\phi(\bx,s,\bt)],
\end{equation} where $\phi(\bx,s,\bt)$ is any formula in $\LL(\aaq)$ and $\bt$ is a finite sequence of countable subsets of $M$.
\end{definition}

On a club determined structure the quantifiers $\stat$ (``stationarily many") and $\aaq$ (``club many") coincide on definable sets. The truth of $\aaqs\phi(s,{\bb},\bt)$ in a structure $\mm$ can be written in the form of a two-person perfect information zero-sum game $G(\phi,\mm,{\bb},\bt)$: the players alternate to pick elements $a_0,a_1,\ldots$ from $M$. After $\omega$ moves Player II wins if $s=\{a_0,a_1,\ldots\}$ satisfies $\phi(s,{\bb},\bt)$ in $\mm$. A structure $\mm$ is club determined if and only if the game $G(\phi,\mm,\bb,\bt)$ is determined for all formulas $\phi$ and all parameters $\bb$. Hence the name.

There are several results in \cite{MR556893} suggesting that club determined structures have a `better' model theory than arbitrary structures. For a start, every consistent first order theory has a club determined model. Moreover, every club determined uncountable model has an $\LL(\aaq)$-elementary submodel of cardinality $\aleph_1$, while for arbitrary structures this cannot be proved in ZFC. It fails if $V=L$ (\cite{MR506381}), but  holds if we assume $\PFA^{++}$ (folklore).
%

\begin{lemma}\label{single}
If a first order structure $\mm$ is {club determined}, then  
 \begin{equation}\label{490002}
\mm\models
\forall\bx[\aaq \bs\phi(\bx,\bs,\bt)\vee
\aaq \bs\neg\phi(\bx,\bs,\bt)],
\end{equation}
 where $\phi(\bx,\bs,\bt)$ is any formula in $\LL(\aaq)$ and $\bt$ is a finite sequence of countable subsets of $M$.
\end{lemma}

\begin{proof}
Suppose $\phi(\bx,s_1,\ldots,s_n,\bt)$ is a formula in $\LL(\aaq)$, $\bt$ is a finite sequence of countable subsets of $M$, and $\bx$ is a finite sequence of elements of $M$. We use induction on $n$. If $n=1$, the claim is true by assumption. Suppose then $n>1$ and $\mm\models\neg\aaqs_1\aaqs_2\ldots\aaqs_n\phi(\bx,s_1,\ldots,s_n,\bt)$. By the assumption (\ref{cd1}),
$$\mm\models
\aaq s_1\neg\aaq s_2\ldots\aaq s_n\phi(\bx,s_1,\ldots,s_n,\bt),$$
whence by the Induction Hypothesis,
$$\mm\models
\aaq s_1\aaq s_2\ldots\aaq s_n\neg\phi(\bx,s_1,\ldots,s_n,\bt).$$
\end{proof}

\begin{definition}
We say that the inner model $\CAA$ is {\em club determined}, or that \emph{Club Determinacy} holds, if every level $(J'_\alpha,\in,{\tr}\rest\alpha)$ in the construction of $\CAA$ is club determined as a first order structure, i.e. for all $\alpha$: 
\begin{equation}\label{cdehto}
(J'_\alpha,\in,{\tr}\rest\alpha)\models
\forall\bx[\aaqs\phi(\bx,\bt,s)\vee
\aaqs\neg\phi(\bx,\bt,s)],
\end{equation} where $\phi(\bx,\bt,s)$ is any formula in $\LL(\aaq)$ and $\bt$ is a finite sequence of countable subsets of $J'_\alpha$.  We say that  $\CAA$ is {\em club determined for $\phi(\bx,\bt,s)$}, or that \emph{Club Determinacy for $\phi(\bx,\bt,s)$} holds, if (\ref{cdehto}) holds (at least) for the formula $\phi(\bx,\bt,s)$.\end{definition}

Intuitively speaking, if $\CAA$ is club determined, its definition is  more 
robust---the quantifier $\aaq$ is  more lax than it would be otherwise, and in consequence, $\CAA$ is a little easier to compute.

We consider Club Determinacy also with the quantifier $\aaq$ interpreted as $\aaqd$. We say that 
$C({\aaq}_{\delta})$ satisfies \emph{$\delta$-Club Determinacy (for $\phi$)} if it satisfies Club Determinacy (for $\phi$) with $\aaq$ replaced by $\aaqd$.

The main technical result of this paper  says that
if there are a proper class of Woodin cardinals, then
 $\CAA$ is club determined (Theorem~\ref{main1}). We prove the same conclusion also under the alternative assumption of $\PFA$ (Theorem~\ref{main2}). In view of the below Theorem~\ref{tepois}   some large cardinal 
 assumption (in $V$ or in an inner model)  is necessary for Club Determinacy. Of course, a
 proper class of measurable cardinals, as in Theorem~\ref{tepois},  is a much weaker assumption than a proper class of Woodin cardinals, and we do not know the exact large cardinal assumption needed here.

\subsection{Club determinacy from  Woodin cardinals}

We are going to prove Club Determinacy in two cases. The first case is a proper class of  Woodin cardinals. This will be the topic of the current section. In the next section we use the assumption $\PFA$. 


Suppose $\delta$ is a Woodin cardinal. We use $\Pw_{<\delta}$ to denote the stationary tower forcing at $\delta$ and $Q_{<\delta}$ to denote the corresponding countable stationary tower forcing. For details concerning the stationary tower we refer to \cite{MR2069032}.

Here is a sketch of the proof of Club Determinacy. We look at the earliest stage at which Club Determinacy fails for $\CAA$. Let us suppose it fails because a set 
\begin{equation}\label{early}
S=\{s\in \Pw_{\omega_1}(J'_\alpha) : (J'_\alpha,\in,{\tr}\rest\alpha)\models\phi(\ba,s,\bt)\}
\end{equation}
is stationary co-stationary, where $\ba$ is a finite sequence of elements of $J'_\alpha$ and $\bt$ is a finite sequence of countable subsets of $J'_\alpha$. We may assume that $\alpha$ and $\phi(s,\bx,\bt)$ are minimal for which this happens. We show with a separate argument that we can assume w.l.o.g. that $|\alpha|=\aleph_1$ and $\delta^1_2 =\omega_2$, where $\delta^1_2=\sup\{\xi : \xi \mbox{ is the length of a }\bm{\Sigma}^1_2 \mbox{- prewell\-ordering}\}$. Let now $\delta$ be a Woodin cardinal. We force with $Q_{<\delta}$ and obtain  the associated generic embedding $j:V\to M\subseteq V[G]$. Recall that $j(\omega_1)=\delta$ and ${}^\omega M\subseteq M$. In $M$ the set $j(S)$ is, in the sense of $M$, the set of $s\in \Pw_{\omega_1}(J'_\gamma)$ such that   $(J'_\gamma,\in,{\tr}\rest\gamma)\models\phi(j(\ba),s,j(\bt))$, where $\gamma=j(\alpha)$. We use the minimality of $\alpha$ 
and $\phi$ to argue that $({J'_\gamma})^M$ is the $\gamma^{\mbox{\tiny th}}$ level, which we denote  $J^*_\gamma$, of the hierarchy of $C({\aaq}_{\delta})$ in $V$. We also show that $j(\ba)$ is an element $a^*$ of $V$, and $j(\bt)$ is an element $t^*$ of $V$. 


We now argue that we can pick a bijection  $h:\omega_1\rightarrow J'_\alpha$ so that also $j(h)\in V$ and it is independent of the  generic $G$. Hence $S^\ast=\{\beta<\omega_1:h[\beta]\in S\}$  is a stationary co-stationary subset of $\omega_1 $. So also $j(S^\ast)$ is in $V$ and $j(S^\ast)$ is independent of $G$. Now we can pick $G$ such that $S^\ast\in G$ , which implies $\omega_1\in j(S^\ast)$ and another generic filter $G$ such that $\omega_1-S^\ast \in G$ which implies $\omega_1\not\in j(S^\ast)$, a contradiction. 
%
 A detailed proof is below in Theorem~\ref{main1}.

The following general fact about forcing will be used below:

\begin{lemma}\label{generalnonsense}Suppose $\delta$ is a regular cardinal, $\oP$ is a forcing notion such that $|\oP|=\delta$, and $G$ is $\oP$-generic.  If $\delta$ is still a regular cardinal in $V[G]$, then for all $N\in V$, every club of  $(\Pw_{\delta}(N))^{V}$ 
is stationary in $(\Pw_{\delta}(N))^{V[G]}$.  
\end{lemma}

\begin{proof} Without loss of generality, $N$ is an ordinal $\beta$. Let $C$ be a club in $(\Pw_{\delta}(N))^{V}$. Suppose $\tau$ is a forcing term for an algebra on $\beta$. Let $\mu$ be a big enough regular cardinal.
We build in $V$ a chain $M_\a, \a<\delta$, of elementary substructures of $H(\mu)^V$ of cardinality $<\delta$ in such a way that $\oP,\tau,\beta\in M_0$, $M_\alpha\in C$, $M_\nu=\bigcup_{\a<\nu}M_\a$ for limit $\nu$, and $\oP\subseteq \bigcup_{\a<\delta}M_\a$.  Let $G$ be $\oP$-generic. Since $\delta$ is regular in $V[G]$, we can construct, in $V[G]$, an ordinal $\g<\delta$ such that if $D\subseteq\oP$ is a dense set in $M_\g$, then $D\cap G\cap M_\g\ne\emptyset$. Now $M_\g\cap\beta\in V$ is closed under the algebraic operations of the value 
$[\tau]_G$ of $\tau$ in $V[G]$.
\end{proof}


The main technical tool in proving the Club Determinacy is the following result about preservation of stationarity in the forcing $Q_{<\lambda}$:

\begin{proposition}\label{svsvg} Suppose that $\lambda$ is Woodin and $G$ is $Q_{<\lambda}$ generic over $V$. If $S\subseteq \lambda$ and $S\in V$ is stationary in $V$ then $S$ is stationary in $V[G]$.
\end{proposition}
\begin{proof} Suppose that $S$ is not stationary in $V[G]$. Let $\tau$ be a $Q_{<\lambda}$ term for a club subset of $\lambda$ forced to be disjoint from $S$. To simplify notation we assume that the maximal condition forces that $\tau\cap S=\emptyset$. For every $\alpha<\lambda$ let $D_\alpha$ be a maximal anti-chain of conditions which force some ordinal $>\alpha$ into $\tau$. For every $\alpha<\lambda$ let $F_\alpha$ be the function defined on $D_\alpha$ such that $F_\alpha(q)$ is the minimal ordinal above $\alpha$ which is forced into $\tau$ by $q$.
 Let $N$ be an elementary substructure of $H(\kappa)$ for a big enough $\kappa$ such that $\langle D_\alpha :\alpha<\lambda\rangle$ and other relevant elements of the proof are in $N$. Also we require that $N\cap\lambda$ is an ordinal $\delta\in S$ and that $V_\delta\subseteq N$. Clearly $V_\delta$ is closed under $F_\alpha$ for every $\alpha<\delta$.

We use the following definition:
\begin{definition} Let $D$  be a maximal anti-chain in $Q_{<\lambda}$. We say that $X\in \Pw_{\omega_1}(V_\lambda)$ \emph{catches} $D$ \emph{below} $\rho$ if   there is $q\in D\cap X\cap V_\rho$ such that $X\cap\cup q\in q$.
\end{definition}

The following definition is a modification to $Q_{<\lambda} $ of definition 2.5.1 of  \cite{MR2069032}.

\begin{definition} Let $D$ be a maximal anti-chain in $Q_{<\lambda}$. We say that $D$ is \emph{semiproper} \emph{at} $\rho$ if for every $X\prec V_{\rho+2}$, $X$ countable,  there is a countable  $Y\prec V_{\rho+2}$ such that $Y$ catches $D$ below $\rho$ and $Y$ end extends $X$ below $\rho$ (i.e. if $\alpha\in (Y-X)\cap\rho$ then 
$\alpha\geq\sup(X\cap\rho)$).
\end{definition}

The following fact follows immediately  from the modification to $Q_{<\lambda}$ of theorem 2.5.9 of \cite{MR2069032}:

\begin{claim} For every $D_\alpha$ there are unboundedly many inaccessible cardinals $\gamma<\lambda$ such that $D_\alpha$ is semiproper at $\gamma$.
\end{claim}

From  $N\prec H(\kappa)$ and $N\cap\lambda=\delta$ it follows that for every $\alpha<\delta$, there are unboundedly many $\gamma<\delta$ such that $D_\alpha$ is semiproper at $\gamma$.

In the following arguments we assume that $V_{\delta+2}$ is also endowed with a fixed well order.

\begin{lemma}\label{sss} For every countable $X\prec V_{\delta+2}$ such that
 $\langle D_\alpha\cap V_\delta :\alpha<\delta\rangle\in X$ there is a countable $Y\prec V_{\delta+2}$ such that $X\subseteq Y$ and  for every $\alpha\in Y\cap\delta$, $Y$ catches $D_\alpha$ below $\delta$.
\end{lemma}
\begin{proof} We define by induction an increasing sequence   $\langle X_n :n<\omega\rangle$ of countable elementary substructures of $V_{\delta+2}$ where $X_0=X$, a sequence 
$\langle \alpha_n:n<\omega\rangle $ of ordinals less than $\delta$  such that $\alpha_n\in X_n$, and an increasing  sequence $\langle \gamma_n : n<\omega\rangle$, $\gamma_n\in X_n$, such that $D_{\alpha_n}$ is semiproper at $\gamma_n$. By dovetailing we make sure that for every $n<\omega$ and $\alpha\in X_n\cap \delta$ there is $k$ such that $\alpha=\alpha_k$. Also we keep the inductive assumption that $X_{n+1}$ catches $D_{\alpha_n} $ below $\gamma_n$ and that it  is an end extension of $X_n$ below $\gamma_n$. So it follows that $X_{n+1}$ continues to catch $D_{\alpha_k}$ below $\gamma_k$ for all $k<n$.

Given $X_n$. Pick $\alpha_n\in X_n$ so as to continue our dovetailing process. Let $\gamma_n$ be an element of $X_n$ above $\gamma_{n-1}$  such that $D_{\alpha_n}$ is semiproper at $\gamma_n$. Such a $\gamma_n$ exists in $X_n$ since $X_n$ is an elementary substructure of $V_{\delta+2}$ and $D_{\alpha_n}\in X_n$. (Recall that $\langle D_\alpha:\alpha<\lambda\rangle \in X\subseteq X_n$.)

Since $X_n\prec V_{\delta+2}$, we have $R=X_n\cap V_{\gamma_n+2}\prec V_{\gamma_n+2}$, hence there is a countable $Z\prec V_{\gamma_n+2}$ such that $R\subseteq Z$, $Z$ is an end extension of $R$ below $\gamma_n$, and $Z$ catches $D_{\alpha_n} $ below $\gamma_n$.   We define $X_{n+1}$ to   be all the elements of $V_{\delta+2}$ which are definable in $V_{\delta+2}$ from $X_n\cup Z$.   Clearly $X_{n+1}\prec V_{\delta+2}$.
\begin{claim} $X_{n+1}\cap V_{\gamma_n}=Z\cap V_{\gamma_n}$.
\end{claim}
\begin{proof} Clearly $Z\cap V_{\gamma_n}\subseteq X_{n+1}\cap V_{\gamma_n}$. For the other direction let $a\in X_{n+1}\cap V_{\gamma_n}$. Then $a$ is definable from some elements $\vec{b}$ of $X_n$    and an element $c$  of $Z$ by a formula $\phi(x,\vec{b},c)$. (It is enough to consider a single element $c$ of $Z$, since $Z$ is closed under forming finite sequences.)  Consider the following function $h:V_{\gamma_n}\rightarrow V_{\gamma_n}$. We let
$h(y)$ to be  the unique element $d$ of $V_{\gamma_n}$ satisfying $\phi(d,\vec{b},y)$ if there is such a unique element, and $0$ otherwise. Now $h\in V_{\delta+2}$ and $h$ is definable in $V_{\delta+2}$ from $\vec{b}$. Moreover, it is a function from $V_{\gamma_n}$ to $V_{\gamma_n}$. So $h\in X_n\cap V_{\gamma_n+2}$.  So $h\in Z$. Clearly $a=h(c)$. Hence $a=h(c)\in Z$.
\end{proof}

Continuing the proof of Lemma~\ref{sss}, it follows from the claim that $X_{n+1}$ end extends $X_n$ below $\gamma_n$ and catches $D_{\alpha_n}$
 From our inductive assumptions it follows that  $X_{n+1}$ catches $D_{\alpha_k}$ below $\gamma_k$ for all $k\leq n$.
Now, if we define $Y=\cup_n X_n$, then $Y$  satisfies the requirements of the lemma.
\end{proof}

We continue the proof of Proposition~\ref{svsvg} with the following:

\begin{claim} The set $T=$
\begin{equation}\label{T}
\begin{array}{lcl}
\{ X\in \Pw_{\omega_1}(V_{\delta+1}) : X\prec V_{\delta+1},
 \mbox{ $X$ catches $D_\alpha$
 below $\delta$ for every } \mbox{$\alpha\in X\cap \delta$}\}
\end{array} 
\end{equation}
  is stationary in $\Pw_{\omega_1}(V_{\delta+1})$.
\end{claim}
\begin{proof}Assume otherwise, then there is a function $g:V_{\delta+1}\rightarrow V_{\delta+1}$ such that every countable  $X\subseteq V_{\delta+1}$ which is closed under $g$ is not in $T$. The function $g$ can be coded as an element of $V_{\delta+2}$. (We use $g$ also for  the code in $V_{\delta+2}$.) Let $X$ be a countable elementary substructure of $V_{\delta+2}$ containing $g$ and the sequences $\langle D_\alpha\cap V_\delta:\alpha<\delta\rangle$.  By the above lemma, there is a countable $X\subseteq Y\prec V_{\delta+2}$ such that $Y$ catches $D_\alpha$ below 
$\delta$  for every   $\alpha\in Y$.  It is obvious that $Y\cap V_{\delta+1}\in T$, but $g\in Y$ so $Y\cap V_{\delta+1}$ is closed under the function $g$, which is a contradiction.
\end{proof}

We can now finish the proof of Proposition~\ref{svsvg}. By the above claim $T\in Q_{<\lambda}$. We claim that $T\force\delta\in \tau$. Suppose that $T'\leq T$ such that $T'$ forces  $\alpha<\delta$ to be  a bound for $\tau\cap \lambda$.  We can assume   without loss of generality that for every $Z\in T'$  $\alpha\in Z$. Since $T'\leq T$ we can also assume that every $Z\in T'$ catches $D_\alpha$ below $\delta$. For $Z\in T'$ let $G(Z)\in Z\cap D_\alpha$  witness the fact that $Z$ catches $D_\alpha$. By Fodor's lemma there is $q\in D_\alpha$ such that the set $T''=\{Z\in T' : G(Z)=q\}$ is stationary in $\cup T'$. But $T''$ forces  $T''\leq q$. But using the function $F_\alpha$ we can see that $q$ forces some ordinal above $\alpha$ to be in $\tau\cap\delta$. We have a contradiction.
\end{proof}

\begin{proposition}\label{late}Suppose $\lambda$ is Woodin and $G$ is $Q_{<\lambda}$-generic over $V$.
For every set $A$ in $V$,  if $S\subseteq \Pw_\lambda(A)$ is stationary in $V$,
 then it is a stationary subset of $\Pw_\lambda(A)$ in $V[G]$.
\end{proposition}

\begin{proof}
Without loss of generality we can assume that $\lambda\subseteq A$ and that for all $X\in S$, 
$X\cap\lambda$ is an ordinal. If $S$ is not stationary in $V[G]$,  there is an algebra $\mathcal{A}\in V[G]$ with countably many operations $\la f_n : n<\omega\ra$ where the arity of $f_k$ is $k_n$, such that no member of $S$ is closed under these operations. Let ${\tau_n}$ be a $Q_{<\lambda}$-name for ${f_n}$. In order to simplify notation we assume that the maximal condition of $Q_{<\lambda}$ forces that no member of $S$ is closed under all the functions ${\tau_n}$. For every $n$ and $k_n$-tuple $\ba$ of members of $A$, let $D_{n,\ba}$ be a maximal antichain of $Q_{<\lambda}$ of conditions which force a value for ${\tau_n}(\ba)$. Now the proof continues as the proof of Proposition~\ref{svsvg}. Especially after the model $N$ has been chosen as an elementary superstructure of a big enough $H(\theta)$ so that $A$, $\tau_n$'s , etc are all in $N$, $A\cap N\in S$, $N\cap\lambda=\delta$ for some $\delta<\lambda$, as well as $\pi:A\cap N\rightarrow \delta$, one proves the following modification  of Lemma~\ref{sss}: 
\begin{quote}
For every countable $X\prec V_{\delta+2}$ such that
 $\langle D_{n,\ba}\cap V_\delta :\ba\in(A\cap N)^{k_n}, n<\omega\rangle\in X$ there is a countable $Y\prec V_{\delta+2}$ such that $X\subseteq Y$ and  $Y$ catches $D_{n,\ba}$ below $\delta$ for every $n<\omega$ and every $\ba\in (A\cap N)^{k_n}$ with $\pi[\ba]\in Y^{k_n}$.
\end{quote} The rest is as in  Proposition~\ref{svsvg}.\end{proof}

 As above, let $\bj'_\alpha=(J'_\alpha,\in,{\tr}\rest\alpha)$ be the hierarchy of $\CAA$ in $V$.
Let $\bj^*_\alpha=(J^*_\alpha,\in,{{\tr}}^*\rest \alpha)$ be the corresponding hierarchy of $C(\aaqd)$ in $V$.
We will compare these two inner models, or rather $C(\aaqd)$ and the image of $\CAA$ under a generic ultrapower embedding.  We use $\ma\models_{\delta}\phi$ to denote $\ma\models\phi$ when we think of $\phi$ as a sentence of $\LL(\aaqd)$ rather than of $\LL(\aaq)$. 

Suppose now $\delta$ is a Woodin cardinal.
Let $G$ be $Q_{<\delta}$-generic  and $j:V\to M\subseteq V[G]$ the generic ultrapower embedding.
Let $\bj''_\alpha=(J''_\alpha,\in,{{\tr}}''\rest \alpha)$ be the hierarchy of $\CAA$ in $M$.
As a part of the proof that $\CAA$ satisfies Club Determinacy we show that this inner model   $\CAA$ in the sense of $M$ is actually the inner model $C(\aaqd)$ in the sense of  $V$ (see Proposition~\ref{thesame}). We show this by a level by level analysis of the two aa-hierarchies $(\bj''_\alpha)$  and 
 $(\bj^*_\alpha)$.

In the subsequent proofs we will use parameters from $V$ although we are dealing also with $M$. Lemma~\ref{omega2} below shows that while $j$ is certainly not definable in $V$, it maps relevant parameters to $V$. First we need an auxiliary result. The following result is a widely known folklore result, but we include a sketch of the proof for the reader's convenience:

\begin{lemma}\label{ww}Assume $x^\#$ exists for every $x\subseteq\omega$. Then 
\begin{equation}\label{876653}
\delta^1_2=\sup\{((\aleph_1^V)^+)^{L[x]} : x\subseteq\omega\}.
\end{equation}\end{lemma}
\begin{proof}
Kunen's proof of the result of Martin (\cite{MR2907001}) to the effect that
every  well founded  $\bsot$ relation has rank $<\omega_2$, shows that 
this rank is actually less than $((\aleph_1^V)^+)^{L[x]}$, where $x$ is the real parameter of the 
$\bsot$-definition. This gives one direction of (\ref{876653}).
For the other direction, suppose  $x$ is a real and 
$\eta=((\aleph_1^V)^+)^{L[x]}$. 
Every ordinal less than $\eta$ is definable in $L[x]$ from some  
$x$-indiscernibles  $\le\aleph_1^V$. This gives the other direction and finishes the proof of
(\ref{876653}). 
%
We can define a relation  between $n$-tuples of reals coding the indiscernibles and the formula. This relation is $\Delta^1_2$ using $x^\#$ as a parameter. The rank of the relation is $\eta$ and therefore $\eta<\delta^1_2$.
\end{proof}

\begin{lemma}\label{omega2}
Assume $\delta^1_2=\omega_2$ and $\delta$ is a Woodin cardinal. Suppose we force with $Q_{<\delta}$ and   the associated generic embedding is $j:V\to M\subseteq V[G]$. Then $j\rest\omega_2\in V$.
In particular, if $s$ is a countable subset of $\omega_2$, then $j(s)\in V$. Moreover, there is $t\in V$ such that $\force_{Q_{<\delta}}j(\check{s})=\check{t}$.
\end{lemma}

\begin{proof} Let, by Lemma~\ref{ww}, $g\in V$ be a function on $\omega_2$ such that for all
$\alpha<\omega_2$, $g(\alpha)$ is a subset of $\omega$ with $\alpha<((\aleph_1^V)^+)^{L[g(\alpha)]}$. 
%
Since $g(\alpha)^\sharp$ exists, there is a term $\tau^{g(\alpha)^\sharp}_\alpha(\vec{x},y)$  such that $\alpha=\tau_\alpha^{g(\alpha)^\sharp}(\vec{\beta},\omega^V_1)$, where $\vec{\beta}<\omega_1^V$.
Note that now $j(\alpha)=\tau_\alpha^{g(\alpha)^\sharp}(\vec{\beta},\delta)$. It follows that $j\restriction\omega_2$ is in $V$.
%
\end{proof}

We now prove the main result of this section:

\begin{theorem}\label{dplau}\label{main1}
If there is a proper class of Woodin cardinals, then $\CAA$ is club determined.
\end{theorem}

\begin{proof}Suppose $\alpha$ is the smallest ordinal for which $(J'_\alpha,\in,{\tr}\rest\alpha)$ fails to satisfy Club Determinacy. We can collapse $|\alpha|$ to $\aleph_1$ without changing $\CAA$ (Proposition~\ref{sigmaclosed}). Hence we may assume w.l.o.g. $|\alpha|=\aleph_1$. 
By a result of Shelah we can, starting from a Woodin cardinal, force  the $\omega_2$-saturation of the non-stationary ideal on $\omega_1$ with semi-proper  forcing. Since $|\alpha|=\aleph_1$, this forcing does not change $\CAA$ up to the level $\alpha+1$. 
Since we have also a measurable cardinal, we may conclude that $\delta^1_2=\omega_2$ (\cite[Theorem 3.17]{MR2723878}). Hence we may assume, w.l.o.g. $\delta^1_2=\omega_2$. 
By Lemma~\ref{ww} there is a real $x$ and $f\in L[x]$ such that $f:\alpha\to \omega_1^V$ is a bijection. Suppose $f$ is the least in the canonical well-order of $L[x]$.
Let  $\delta>\alpha$ be a Woodin cardinal and $j:V\to M\subseteq V[G]$ as above. Let $\gamma=j(\alpha)$. Now $g=j(f):\gamma\to\delta$ is a bijection and the $L[x]$-least such.
Clearly, $g\in V$.

Suppose $\phi$ witnesses the failure of Club Determinacy of $J'_\alpha$, i.e. there is a  stationary co-stationary set
\begin{equation}\label{P}
P=\{s\in \Pw_{\omega_1}(J'_\alpha):(J'_\alpha,\in,{\tr}\rest\alpha)\models\phi(\ba,s,t)\},
\end{equation}
where $\ba\in J'_\alpha$ and $t$ is a countable subset of $J'_\alpha$. We may assume that $\phi$ is minimal for such   an $\ba\in J'_\alpha$ and $t$, a countable subset of $J'_\alpha$, to exist. 
%
%
%
%
 By the elementarity of $j$, 
$$Q=_{\mbox{\tiny def}}j(P)$$ is a counter-example to Club Determinacy of $(J''_\gamma,\in,{{\tr}}''\rest\gamma)$ in $M$ in the sense that  
\begin{equation}\label{Q}
Q=\{s\in (\Pw_{\delta}(J''_\gamma))^M:((J''_\gamma,\in,{{\tr}}''\rest\gamma)\models\phi(j(\ba),s,j(t)))^M\},
\end{equation}
where  $j(t)\in (\Pw_{\delta}(J''_\gamma))^M$, is stationary co-stationary in $M$.  Moreover, 
$\gamma$ is minimal such that Club Determinacy  fails in $(J''_\gamma,\in,{{\tr}}''\rest\gamma)$ in $M$, and $\phi$ is minimal such that Club Determinacy fails for $\phi$ in $(J''_\gamma,\in,{{\tr}}''\rest\gamma)$ in $M$ with some parameters $\ba$ and $t$.  In other words, 
\begin{equation}\label{minimal}
\left\{\begin{array}{ll}
(a)&\mbox{If $\eta\in\gamma\cap\Lim$, then $(J''_\eta,\in,{{\tr}}''\rest\eta)$ satisfies}\\
&\qquad \mbox{ Club Determinacy in $M$,}\\
(b)&\mbox{If $\psi$ is a subformula of $\phi$, then
$(J''_\gamma,\in,{{\tr}}''\rest\gamma)$ satisfies}\\
&\qquad \mbox{ Club Determinacy for $\psi$ in $M$,}\\
\end{array}\right.
\end{equation}

We need an auxiliary concept relevant only for this proof: 
Let us say that $(J^*_\xi,\in,{{\tr}}^*\rest\xi)$ satisfies \emph{weak $\delta$-Club Determinacy (for $\psi$)} if it satisfies $\delta$-Club Determinacy (for $\psi$) with the restriction that the parameters ($\bt$ in Definition~\ref{cd}) are  in $V\cap M$.



\begin{lemma}\label{tomorrow}
For all $\eta\in\gamma\cap\Lim$, $\bj''_\eta=\bj^*_\eta$ and $(J^*_\eta,\in,{{\tr}}^*\rest\eta)$ satisfies  weak $\delta$-Club Determinacy (in $V$).
\end{lemma}

\begin{proof}
We prove the claim by induction on limit ordinals $\eta$. 

\subsubsection*{Successor case}

Let us assume the Lemma for $\eta\in\gamma\cap\Lim$ and prove it for $\eta+\omega<\gamma$. 
By definition,
$$\bj^*_{\eta+\omega}=(J^*_{\eta+\omega},\in,{{\tr}}^*\rest {\eta+\omega}),$$
where $$J^*_{\eta+\omega}=\rud_{\tr^*}(J^*_\eta\cup\{J^*_\eta\}).$$ 
By Induction Hypothesis, 

$$\begin{array}{lcl}
{{\tr}}^*\rest\eta+\omega
&=&\bigcup_{\delta\in\eta\cap\Lim}(\{\delta\}\times{{\tr}}^*_\delta)\\
&=&\bigcup_{\delta\in\eta\cap\Lim}(\{\delta\}\times{{\tr}}''_\delta)\\
&=&{{\tr}}''\rest\eta+\omega.\end{array}$$
and therefore
$$\begin{array}{lcl}
J^*_{\eta+\omega}&=&\rud_{\Tr^*}(J^*_\eta\cup\{J^*_\eta\})\\
&=&\rud_{\Tr''}(J''_\eta\cup\{J''_\eta\})\\
&=&J''_{\eta+\omega}.
\end{array}$$ 

Next we prove ${{\tr}}^*_{\eta+\omega}={{\tr}}''_{\eta+\omega}$.
To this end, let $N=J^*_{\eta+\omega} =J''_{\eta+\omega}$ and $R={{\tr}}^*\rest\eta+\omega={{\tr}}''\rest\eta+\omega$. We prove that the following  equivalence holds,\footnote{Note that $M^{<\delta}\subseteq M$ in $V[G]$, whence $(\Pw_{\delta}(N))^V\subseteq (\Pw_{\delta}(N))^M$.} whenever $\psi(\ba,\bt)$ is an $\LL(\aaq)$-formula:

\begin{equation}\label{equiv}
\left\{\begin{array}{ll}
(a)&\mbox{If $\ba\in N (\subseteq M)$ and $\bt$ in  $(\Pw_{\delta}(N))^V\cap M$, then}\\
&((N,\in,R)\models_\delta\psi(\ba,\bt))^V \iff ((N,\in,R)\models\psi(\ba,\bt))^M.\\
(b)&\mbox{$(N,\in,R)$ has  weak $\delta$-Club Determinacy for $\psi(\ba,\bt)$ (in $V$).}\\
\end{array}\right.
\end{equation}

The claim  ${{\tr}}^*_{\eta+\omega}={{\tr}}''_{\eta+\omega}$ follows from (\ref{equiv}) by forgetting the parameter $\bt$, which, however, is important for the success of the inductive proof of (\ref{equiv}). By the nature of this inductive proof, it suffices to consider the restricted case $\bt\in(\Pw_{\delta}(N))^V\cap M$. Of course, the right hand side of the equivalence in (\ref{equiv}a) makes sense only if $\bt\in M$. Respectively the left hand side requires $\bt\in V$. Therefore it is reasonable to assume in
(\ref{equiv}a) that $\bt\in V\cap M$. This is also the assumption in weak $\delta$-Club Determinacy.

We prove  the conditions (\ref{equiv}) by induction on $\psi$. It suffices to prove the induction step for the $\aaq$-quantifier. Thus we assume  (\ref{equiv}) for $\psi=\theta(\ba,s,\bt)$ and prove (\ref{equiv}) for $\psi(\ba,\bt)=\aaqs\theta(\ba,s,\bt)$.

To prove ``$\Rightarrow$" in (\ref{equiv}a)  for $\psi(\ba,\bt)=\aaqs\theta(\ba,s,\bt)$, {suppose} $$((N,\in,R)\models_\delta\aaqs\theta(\ba,s,\bt))^V.$$ Let $K\in V$, $K\subseteq\Pw_\delta(N)^V$, be a club of $s$ such that 
\begin{equation}\label{tahti}
((N,\in,R)\models_\delta\theta(\ba,s,\bt))^V.
\end{equation} By Lemma~\ref{generalnonsense} $K$ is stationary in $V[G]$.
If $((N,\in,R)\not\models\theta(\ba,s,\bt))^M$, then by (\ref{minimal}a)  there is a club $H$ of $s\in\Pw_{\omega_1}(N)^M$ such that $((N,\in,R)\not\models\theta(\ba,s,\bt))^M.$ Since ${}^\omega M\subseteq M$, this club $H$ is also a club in $V[G]$. Let $s\in K\cap H$. Note that $s\in M$. By Induction Hypothesis, $((N,\in,R)\not\models\theta(\ba,s,\bt))^V,$ contrary to (\ref{tahti}). Thus $((N,\in,R)\models\theta(\ba,s,\bt))^M$.

To prove ``$\Leftarrow$", {suppose} $$((N,\in,R)\models\aaqs\theta(\ba,s,\bt))^M.$$ Let $K\in M$, $K\subseteq\Pw_\delta(N)^M$, be a club of $s$ such that 
\begin{equation}\label{tahtia}
((N,\in,R)\models\theta(\ba,s,\bt))^M.
\end{equation}Since ${}^\omega M\subseteq M$, this club $K$ is also a club in $V[G]$.
If $((N,\in,R)\not\models_\delta\aaqs\theta(\ba,s,\bt))^V$, then by the weak $\delta$-Club Determinacy 
for $\theta(\ba,s,\bt)$ of $(N,\in,R)$ in $V$ there is a club $H$ of $s\in\Pw_{\delta}(N)^V$ such that $((N,\in,R)\not\models_\delta\theta(\ba,s,\bt))^V.$   By Lemma~\ref{generalnonsense} $H$ is stationary in $V[G]$. Let $s\in K\cap H$. Note that $s\in M$. By Induction Hypothesis, $((N,\in,R)\not\models\theta(\ba,s,\bt))^M,$ contrary to (\ref{tahtia}).

We move to proving (\ref{equiv}b) for $\psi(\ba,\bt)=\aaqs\theta(\ba,s,\bt)$. Let $\bt=(u,t_1,\ldots,t_n)$ and $\vec{t'}=(t_1,\ldots,t_n)$. We need to prove
\begin{equation}\label{wcd}
N\models_\delta \aaqu\psi(\ba,u,\vec{t'})\vee  \aaqu\neg\psi(\ba,u,\vec{t'}),
\end{equation} 
where $\ba\in N$ and $\vec{t'}\in (\mathcal{P}_\delta(N))^V\cap M$.

By (\ref{minimal}a), 
$$(N\models \aaqu\psi(\ba,u,\vec{t'})\vee  \aaqu\neg\psi(\ba,u,\vec{t'}))^M.$$
For example, there is a club $K$ in $M$ of countable subsets $u$ of $N$ such that
$(N\models\psi(\ba,u,\vec{t'}))^M$. The set $K$ is still club in $V[G]$. We can now argue that $N\models_\delta\aaqu\psi(\ba,u,\vec{t'})$, for otherwise there is a stationary set $U$ of elements of $(\mathcal{P}_\delta(N))^V$ such that $N\models\neg\psi(\ba,u,\vec{t'})$. By Theorem~\ref{late} the set $U$ is stationary in $V[G]$. Intersecting $K$ and $U$ leads to a contradiction with (\ref{equiv}a). The argument is essentially the same if  there is a club $K$ in $M$ of countable subsets $u$ of $N$ such that
$(N\models\neg\psi(\ba,u,\vec{t'}))^M$. We have proved (\ref{wcd}).

This ends the proof on (\ref{equiv}) and ends the successor case. 

\subsubsection*{Limit case}

Let us assume $\nu<\gamma$ is a limit of limit ordinals and the claim of the Lemma holds for $\eta\in\nu\cap\Lim$. Now  we show that it  holds for $\nu$, too.

By Induction Hypothesis,
$$\begin{array}{lclclcl}
J^*_{\nu}&=&\bigcup_{\eta\in\nu\cap\Lim}J^*_\eta&=&\bigcup_{\eta\in\nu\cap\Lim}J''_\eta&=&J''_\nu.\\ 
{{\tr}}^*\rest\nu
&=&\bigcup_{\eta\in\nu\cap\Lim}{{\tr}}^*\rest\eta&=&\bigcup_{\eta\in\nu\cap\Lim}{{\tr}}''\rest\eta&=&{{\tr}}''\rest\nu.
\end{array}$$
Next we note that  ${{\tr}}^*_{\nu}={{\tr}}''_{\nu}$ can be proved with exactly the same  argument as above for 
${{\tr}}^*_{\eta+\omega}={{\tr}}''_{\eta+\omega}$.
This ends the proof for the limit case. %
\end{proof}

\begin{lemma}\label{tomorroww}
$J''_\gamma=J^*_\gamma$, ${{\tr}}''\rest\gamma={{\tr}}^*\rest\gamma$, and letting $N=J''_\gamma$ and $R={{\tr}}''\rest\gamma$, the equivalence (\ref{equiv}a)  holds for  $\phi$\footnote{$\phi$ is the minimal counter-example chosen in the beginning of the proof.} in place of $\psi$.
\end{lemma}

\begin{proof}Clearly, by Lemma~\ref{tomorrow}, $J''_\gamma=J^*_\gamma$ and ${{\tr}}''\rest\gamma={{\tr}}^*\rest\gamma$.
We now prove (\ref{equiv}a) with $\gamma$ in place of $\xi$ by induction on  subformulas $\psi$ of $\phi$. It suffices to prove the induction step for the $\aaq$-quantifier. Thus we assume  (\ref{equiv}a) for $\psi=\theta(\ba,s,\bt)$ and prove (\ref{equiv}a) for $\psi=\aaqs\theta(\ba,s,\bt)$.

To prove ``$\Rightarrow$", {suppose} $((N,\in,R)\models_\delta\aaqs\theta(\ba,s,\bt))^V.$ Let $K\in V$, $K\subseteq\Pw_\delta(N)^V$, be a club of $s$ such that 
\begin{equation}\label{tahtib}
((N,\in,R)\models_\delta\theta(\ba,s,\bt))^V.
\end{equation} By Lemma~\ref{generalnonsense} $K$ is stationary in $V[G]$.
If $((N,\in,R)\not\models\theta(\ba,s,\bt))^M$, then by (\ref{minimal}b) there is a club $H$ of $s\in\Pw_{\omega_1}(N)^M$ such that $((N,\in,R)\not\models\theta(\ba,s,\bt))^M.$ Since ${}^\omega M\subseteq M$, this club $H$ is also a club in $V[G]$. Let $s\in K\cap H$. By Induction Hypothesis, $((N,\in,R)\not\models\theta(\ba,s,\bt))^V,$ contrary to (\ref{tahtib}).

To prove ``$\Leftarrow$", {suppose} $((N,\in,R)\models\aaqs\theta(\ba,s,\bt))^M.$ Let $K\in M$, $K\subseteq\Pw_\delta(N)^M$, be a club of $s$ such that 
\begin{equation}\label{tahtic}
((N,\in,R)\models\theta(\ba,s,\bt))^M.
\end{equation}Since ${}^\omega M\subseteq M$, this club $K$ is also a club in $V[G]$.
If $((N,\in,R)\not\models_\delta\aaqs\theta(\ba,s,\bt))^V$, then by the weak $\delta$-Club Determinacy for $\theta$ of $(N,\in,R)$ in $V$, which is part of our Induction Hypothesis, there is a club $H$ of $s\in\Pw_{\delta}(N)^V$ such that $((N,\in,R)\not\models_\delta\theta(\ba,s,\bt))^V.$   By Lemma~\ref{generalnonsense} $H$ is stationary in $V[G]$. Let $s\in K\cap H$. By Induction Hypothesis, $((N,\in,R)\not\models\theta(\ba,s,\bt))^M,$ contrary to (\ref{tahtic}).

This ends the proof on (\ref{equiv}).
\end{proof}

Recall the definition of $P$ in (\ref{P}) and of $Q$ in (\ref{Q}). By Lemma~\ref{omega2} there are $\vec{a^*}\in V$ and ${t^*}\in V$ such that $j(\vec{a})=\vec{a^*}$ and $j(t)={t^*}$. Moreover, $\vec{a^*}$ and ${t^*}$ are independent of $G$. 

  The mapping   $f$ was defined as a bijection of $\alpha$ onto $\omega_1^V$. There is a  bijection of $J'_\alpha$ onto $\alpha$, definable over $J'_\alpha$. By combining the two bijections, we get a bijection $\tilde{f}:J'_\alpha\rightarrow \omega_1^V$. Similarly we get in   bijection of $J''_\gamma$ onto $\gamma$, definable over $j(J'_\alpha)=J''_\gamma=J^\ast_\gamma$. Since $J''_\gamma\in V$ this bijection is in $V$ and by combining it with $g$ we get a bijection $\tilde{g}:J''_\gamma\rightarrow \delta$. Since  $j(f)=g\in V$, then also $\tilde{g}=j(\tilde{f})$, it is  in $V$, and it is independent of the generic $G$.

By Lemma~\ref{tomorroww} and (\ref{equiv})
\begin{equation}\label{tomorrowww}
Q\cap V=\{s\in \Pw_{\delta}(J^*_\gamma):(J^*_\gamma,\in,{{\tr}}^*\rest\gamma)\models\phi(\vec{a^*}\,s,\bar{t^*})\}
\end{equation}
and therefore $Q\cap V\in V$, and the identity (\ref{tomorrowww}) holds independently of $G$. 
Let $$S=\{\beta<\omega_1:\tilde{f}^{-1}[\beta]= s\mbox{ for some }s\in P\}.$$
It is easy to see that $S$ is stationary co-stationary on $\omega_1$. Note that the set
$$j(S)=\{\beta<\delta:\tilde{g}^{-1}[\beta]= s\mbox{ for some }s\in j(P) \},$$ 
$$=\{\beta<\delta:g^{-1}[\beta]= s\mbox{ for some }s\in j(P)\cap V (=Q\cap V)\},$$  is in $V$, and is independent of $G$.
Let $G_1$ be $Q_{<\delta}$-generic such that $S\in G_1$ and let $j_1:V\to M_1$ be the associated embedding.
Let $G_2$ be $Q_{<\delta}$-generic such that $\omega_1\setminus S\in G_2$  and let $j_2:V\to M_2$ be the associated embedding.
Now $\omega_1\in j_1(S)$ and $\omega_1\notin j_2(S)$. But, $j_1(S)=j_2(S)$, a contradiction.\end{proof}

\begin{corollary}
Suppose there is a supercompact cardinal. Then Club Determinacy holds.
\end{corollary}

\begin{proof} Suppose $\kappa$ is supercompact.
 Let $\a$ be the least such that $(J'_\a,\in,{\tr}\rest \alpha)$ is not club determined. Since $V_\kappa\prec_2 V$, we can assume $\a<\kappa$. Since $\kappa$ is a limit of Woodin cardinals, we can proceed as above. 
\end{proof}




\begin{proposition}[\cite{MR2499432}]\label{pfa}
Assuming $\PFA$, there is, for every set $X$, an inner model with a proper class of Woodin cardinals, containing $X$.
\end{proposition}

\begin{proof}We modify Theorem 0.1 of \cite{MR2499432} as follows. Suppose $X$ is an arbitrary set of ordinals, e.g. $X\subseteq\delta$. Let an $X$-\emph{mouse} be a mouse as in \cite{MR2499432} except that the mouse is assumed to contain $X$ and, moreover, it is  required that all the extenders on the coherent sequence have the critical point above $\delta$.
With this modification the proof of Theorem 0.1 in \cite{MR2499432} gives the result that if $\Box(\kappa)$ and $\Box_{\kappa^+}$ fail for  some $\kappa>\delta$, a consequence of $\PFA$, then there is an inner model with a proper class of Woodin cardinals containing  $X$.\end{proof}

\begin{theorem}\label{main2}
Assuming $\PFA$,  Club Determinacy holds.
\end{theorem}

\begin{proof}
Suppose Club Determinacy fails at $(J'_\alpha,\in,\tr\rest\alpha)$ and $\alpha$ is minimal.    
%
Let $X$ contain everything we need for the failure of Club Determinacy, e.g. $X=V_{\omega_2}$.  By Proposition~\ref{pfa} there is an inner model $M$ with a proper class of Woodin cardinals such that $M$ contains $X$. By the choice of $X$, $M$  fails to satisfy Club Determinacy. But this contradicts Theorem~\ref{main1}.
\end{proof}


\section{Applications of Club Determinacy}

We give three types of applications of Club Determinacy. The first is the immediate consequence that uncountable cardinals are measurable in $\CAA$. Our large cardinal assumption in the proof of Club Determinacy was a proper class of Woodin cardinals, so we are far from an optimal result. Our second application is the forcing absoluteness of the theory of $\CAA$. Here we assume  a proper class of Woodin cardinals and use Club Determinacy merely as a tool in the proof. Our third and more substantial application is a proof of $\CH$ in $\CAA$, using Club Determinacy.   

\subsection{Large cardinals}

Recall that, assuming a proper class of Woodin cardinals, uncountable cardinals are Mahlo in $C^*$, and even weakly compact above $\aleph_1$. In \cite{kmv} we were not able to prove that there are measurable cardinals in $C^*$ under any assumption, even consistently. For the presumably bigger inner model $\CAA$ we establish now the measurability of all uncountable regular cardinals. As it turns out, the proof is an immediate consequence of Club Determinacy.

\begin{theorem}\label{tepois}
Suppose $\CAA$ is club determined. Then every regular $\kappa\ge\aleph_1$ is measurable in $\CAA$.
\end{theorem}

\begin{proof} For  $\alpha$ big enough for $J'_\alpha$ to contain all subsets of $\kappa$ in $\CAA$, consider the normal filter:
$$\F=\{X\subseteq\kappa : X\in J'_\alpha, (J'_\alpha,\in,\tr\rest\alpha)\models \aaqs(\sup(s\cap\kappa)\in X)\}.$$
%
Suppose $X\subseteq\kappa$ is in $\CAA$.  Since $(J'_\alpha,\in,\tr\rest\alpha)$ is club determined,
$$(J'_\alpha,\in,\tr\rest\alpha)\models\aaqs(\sup(s\cap\kappa)\in X)\mbox{ or }$$ 
$$(J'_\alpha,\in,\tr\rest \alpha)\models\aaqs(\sup(s\cap\kappa)\notin X).$$
In the first case $X\in\F$. In the second case $\kappa\setminus X\in \F$.
\end{proof}

It remains open whether Club Determinacy, or some reasonable stronger assumption, implies that 
there are higher measurable cardinals in $\CAA$. By Corollary~\ref{nowoodin} below, we cannot hope to have Woodin cardinals in $\CAA$ as a consequence of some large cardinal assumptions. It remains open what happens to singular cardinals. Are they regular, or even large cardinals in $\CAA$?

\subsection{Forcing absoluteness}

 The first order theory of $L(\oR)$ is absolute under set forcing, assuming a proper class of  Woodin cardinals. With a stronger assumption the same is true of the Chang model $C(\LL_{\omega_1\omega_1})$. We can prove the absoluteness of $\CAA$ under set forcing assuming a proper class of Woodin cardinals.  

\begin{proposition}\label{thesame}
Suppose club-determinacy holds, $\delta$ is Woodin,  $G\subseteq  Q_{<\delta}$ is generic and $M$ is the associated generic ultrapower. Then $\CAA^M=C(\aaqd)^V$.
 Hence $C(\aaqd)^V$ satisfies club-determinacy.
\end{proposition}

\begin{proof}The proof is similar elements to the proof of Theorem~\ref{main1}. Let $j:V\to M$ be the elementary embedding associated with $G$. Note that since we assume Club Determinacy in $V$ for  $J'_\alpha$ for all $\alpha$, we have Club Determinacy in $M$ for $J''_\alpha$ for all $\alpha$.
We show by induction on $\alpha$ and on the $\LL(\aaq)$-formula $\phi(s,\bz,\by)$ that if we denote  $N=(J''_\alpha)^M$ and assume, as part of the Induction Hypothesis,  that $N=(J^*_\alpha)^V$, then for every  $\bb\in N$ and $\bt \in \Pw_\delta(N)\cap V \cap M$ (note that $\omega_1^M=\delta$) the following are equivalent. 
\begin{description}
\item[(A1)] $(N\models_{\delta}\aaqs\phi(s,\bt,\bb))^V$.
\item[(A2)] $(N\models \aaqs\phi(s,\bt,\bb))^M$.
\end{description}

Suppose first  (A1). Let $C$ be a club in $V$ of  sets $s\in \Pw_\delta(N)$ satisfying $(N\models_{\delta}\phi(s,\bt,\bb))^V$. By (\ref{generalnonsense}), $C$ is stationary in $V[G]$. Suppose (A2) fails. Then by Club Determinacy in $M$, there is a club $K$ in $M$ of countable $s$ such that $(N\models \neg\phi(s,\bt,\bb))^M$. Since ${}^{\omega} M\subseteq M$, the set $K$ is still club in $V[G]$. Let $s\in K\cap C$. Note that $s\in M$. By Induction Hypothesis, $(N\models_{\delta}\neg\phi(s,\bt,\bb))^V$,
a contradiction.

For the other direction: suppose (A2) i.e.  in $M$ there is a club $D$ of countable sets $s$ such that  $(N\models\phi(s,\bt,\bb))^M$. This $D$ is still club in $V[G]$. Suppose that (A1) fails. Hence the set $S$ of $s\in \Pw_\delta(N)$ in $V$ that satisfy $(N\models_{\delta}\neg\phi(s,\bt,\bb))^V$ is stationary in $V$. By
 (\ref{late}) it is stationary in $V[G]$. Let $s\in D\cap S$. Now $s\in \Pw_\delta(N)\cap V\cap M$, so we have a contradiction with the Induction Hypothesis.
\end{proof}


\begin{theorem}\label{abs}
Suppose there are a proper class of Woodin cardinals. Then the first order theory of $\CAA$ is (set) forcing absolute. 
\end{theorem}

\begin{proof}
Suppose $\oP$ is a forcing notion and $\delta$ be a Woodin cardinal $>|\oP|$. Let $j:V\to M$ be the associated elementary embedding. By Proposition~\ref{thesame} we can argue
$$\CAA\equiv (\CAA)^M=C(\aaqd).$$
On the other hand, let $H\subseteq\oP$ be generic over $V$. Then $\delta$ is still Woodin
in $V[H]$, so we have the associated elementary embedding $j':V[H]\to M'$. By Proposition~\ref{thesame} we can again argue
$$(\CAA)^{V[H]}\equiv (\CAA)^{M'}=(C(\aaqd))^{V[H]}.$$
%
 Using the fact that $|\mathcal{P}|<\delta$ and  that both $C(\aaqd)^V$ and $C(\aaqd)^{V[H]}$ satisfy club determinacy  one can show by induction on $\alpha$ that 
 $$(J^*_\alpha)^V=(J^*_\alpha)^{V[H]}.$$ It follows that $$(C(\aaqd))^{V[H]}=C(\aaqd)^V.$$ 
 
Hence $$(\CAA)^{V[H]}\equiv \CAA.$$
\end{proof}


\subsection{The Continuum Hypothesis}

The fact (Theorem~\ref{abs}) that under large cardinal hypotheses the theory of $\CAA$ is forcing absolute, strongly suggest that we should be able to determine the truth value of the Continuum Hypothesis in $\CAA$. 
Indeed, in this section we  use Club Determinacy to \emph{prove} the Continuum Hypothesis in $\CAA$ (Theorem~\ref{CHH} below). The proof uses the auxiliary concepts of an aa-mouse and an aa-ultrapower, which have hopefully also other uses in the study of $\CAA$. For example we use them below  to  prove 
also $\bDiamond$ in $\CAA$. Our method yields $2^\kappa=\kappa^+$ for $\kappa\le\omega_1^V$ in $\CAA$. (Recall that $\omega_1^V$ is a measurable cardinal in $\CAA$.) Our method seems to yield also full $\GCH$ in $\CAA$\footnote{See footnote 14.}.
   Our previous paper $\cite{kmv}$ gives the consistency of the failure of $\CH$ in $C^*$ relative to the consistency of ZFC. This result extends to $\CAA$ (see Theorem~\ref{slllk} above).
\medskip

\noindent\emph{Convention:} In the rest of this Section we assume Club Determinacy.


\subsubsection{aa-premice}

Our proof uses a new inner model concept which we call aa-premouse. Roughly speaking, an aa-premouse is a pair $(M,T^*)$, where $M$ is a model and $T^*$ is an $\LL(\aaq)$-theory. Intuitively, but not 
in reality, $T^*$ is the $\LL(\aaq)$-theory of $M$. Here $M$ can very well by countable. In countable domains the $\aaq$-quantifier is eliminable, so in general we do not assume $M$ to be a model of $T^*$. Rather, $M$ is a model that has \emph{potential} to become a model of $T^*$.  We fulfil this potential  by building an $\omega_1^V$-chain of elementary extensions of $M$ with the idea that in the limit the theory $T^*$ is really true. For this purpose we define an ultrapower construction---called the aa-ultrapower---for aa-premice. It allows us to iterate a well-chosen countable aa-premouse (iterable aa-premice are called aa-mice) to a big uncountable aa-premouse $(M_{\omega_1},T_{\omega_1}^*)$ where  $T_{\omega_1}^*$ is an $\LL(\aaq)$-theory that is actually true in $M_{\omega_1}$. 

We use the concepts of aa-premouse and aa-ultrapower to prove $\CH$ in $\CAA$. The proof is reminiscent of  Silver's proof of $\GCH$ in $L^\mu$ \cite{MR0278937}. Since we assume Club Determinacy, $\omega_1^V$ is actually a measurable cardinal in $\CAA$. Thus from the point of view of $\CAA$ we start with a countable premouse and iterate it a measurable cardinal times.

We fix the following notation: $\tau_\xi=\{\rR_{\in},\rR_T,\rR_{T^*}\}\cup\{\rP_\eta:\eta<\xi\}$,
$\tau^-_\xi=\tau_\xi\setminus\{\rR_{T^*}\}$. Here $\rR_{\in}$ and $\rR_T$ are binary and $\rR_{T^*},\rP_\eta$ ($\eta<\xi$), are unary. We use $(P)_\xi$ to denote a sequence $\la P_\eta:\eta<\xi\ra$.

\begin{definition}\label{prem}
An \emph{{aa-premouse}} is a structure 
$$\bj^T_\alpha=(J^T_\alpha,\in,T,T^*,(P)_\xi)$$  in the vocabulary $\tau_\xi$ 
%
%
 such that
\begin{enumerate}

\item [(1)] $T\subseteq \alpha\times \LL(\aaq)\times J^T_\alpha$, and  for all\footnote{To simplify notation we use $\ba$ to denote $c_{a_1},\ldots, c_{a_n}$.} 
$\beta<\alpha$, the set $$T_\beta=\{\phi(\ba):(\beta,\phi(\ba))\in T, \ba\in J^T_\beta\}$$ is a complete consistent $\LL(\aaq)$-theory in the vocabulary $\tau_0^-$ extending the first order theory of
$(J^T_\beta,\in,T\rest \beta)$, where  we allow constants $c_a$ for $a\in J^T_\beta$.
\item [(2)] $T^*$ is a complete consistent $\LL(\aaq)$-theory  in the vocabulary $\tau_\xi^-$ extending the first order theory of 
$(J^{T}_\alpha,\in,T,(P)_\xi)$ with constants $c_a$ for $a\in J^{T}_\alpha$.
\item [(3)] $\la P_\eta:\eta<\xi\ra$ is a continuously increasing sequence of subsets of $J^T_\alpha$
and   $\aaqs \forall x(\rP_\eta(x)\to x\in s)\in T^*$, if $\eta<\xi$.
\item [(4)] If $\exists x\phi(x,\ba)\in T^*$, then there is $b\in J^{T}_\alpha$ such that $\phi(c_b,\ba)\in T^*$,  whenever $\phi(\vec{x})$ is an $\LL(\aaq)$-formula in the vocabulary $\tau_\xi^-$ and $\ba\in J^{T}_\alpha$.
\item [(5)] The sentence $$\aaqss\hspace{1pt}\exists x\phi(x,\vs,\ba)\to\aaqss\hspace{1pt}\exists x(\phi(x,\vs,\ba)\wedge\forall y\prec x\neg\phi(y,\vs,\ba))$$ is in $T^*$,  whenever $\phi(x,\vs,\by)$ is an $\LL(\aaq)$-formula in the vocabulary $\tau_\xi^-$ and $\ba\in J^{T}_\alpha$. (For the definition of $\prec$, see Lemma~\ref{well-order}.)
\item [(6)] The Club Determinacy schema 
\begin{equation}\label{cds}
\aaq \bt(\aaqs\phi(\ba,s,\bt\hspace{2pt})\vee
\aaqs\neg\phi(\ba,s,\bt\hspace{2pt})),
\end{equation} where $\phi(\ba,s,\bt)$ is in $\LL(\aaq)$  in the vocabulary $\tau_\xi^-$ and $\ba\in J'_\alpha$, is contained in $T^*$. 
\item [(7)] The sentences $\aaqs\exists x\neg x\in s$ and $\aaqs (\omega\subseteq s)$ are in $T^*$.

\item [(8)] If $\beta\in\alpha\cap\Lim$,  $\phi(\by)$ is an $\LL(\aaq)$-formula in the vocabulary $\tau_0^-$, $\bb\in \bj^T_\beta$, and $\phi(\bb)\in T_\beta$, then 
$\phi(\bb)^{(J^T_\beta)}\in T^*.$
\item [(9)] If $\phi(s,x,\by)$ is an $\LL(\aaq)$-formula in the vocabulary $\tau_\xi^-$ and $\ba\in J^T_\alpha$ such that $\aaqs\exists x\phi(s,x,\ba)\in T^*,$ then $\aaqs\exists x\phi(s,x,{\ba})\to
\aaqs\phi(s,f_{\phi(s,x,{\ba})}(s),{\ba})$ is in $T^*$.
Here we use the term $f_{\phi(s,x,{\ba})}(s)$ to denote the $\prec$-minimal $x$ intuitively satisfying $\phi(s,x,{\ba})$, i.e. we work in a conservative extension of $T^*$, denoted also $T^*$, which contains:
$$\aaqs\exists x\phi(s,x,{\ba})\to
\aaqs(\phi(s,f_{\phi(s,x,{\ba})}(s),{\ba})\wedge$$
$$\forall z(z\prec f_{\phi(s,x,{\ba})}(s)\to\neg\phi(s,z,{\ba}))).$$

\end{enumerate}\end{definition}
Condition 5 simply says that if we can find, for a club of $s$, an $x$ such that $\phi(s,x,\by)$, then for a club of $s$ we can find a $\prec$-\emph{minimal} $x$ such that $\phi(s,x,\by)$. This assumption allows us to have, in a sense,  definable Skolem-functions. Conditions (8)-(9) establish important \emph{coherence} between the predicates $T$ and $T^*$.

\begin{lemma}
If $(J^T_\alpha,\in,T,T^*,(P)_\xi)$ is an aa-premouse and $\beta\in\alpha\cap\Lim$, then 
$$(J^T_\beta,\in,T\cap J^T_\beta,T^*\cap J^T_\beta,(P\cap J^T_\beta)_\xi)$$ is an aa-premouse and $\bJ^T_\beta=\bJ^{T\cap J^T_\beta}_\beta$.
\end{lemma}

In harmony with Lemma~\ref{single} we now prove that Club Determinacy holds in an aa-premouse also for nested $\aaq$-quantifiers: 
 
\begin{lemma}\label{2433}
$\aaq \bt(\aaq \bs\phi(\ba,\bs,\bt\hspace{2pt})\vee
\aaq \bs\neg\phi(\ba,\bs,\bt\hspace{2pt}))\in T^*,$ where $\phi(\ba,\bs,\bt\hspace{2pt})$ is in $\LL(\aaq)$  in the vocabulary $\tau_\xi^-$.
\end{lemma}

\begin{proof}
We use induction on the length $n$ of $\bs$. For $n=1$ the claim is true by definition. Let us then assume the claim for $n$ and prove it for $n+1$. Let $\psi$ be the formula $\aaqs_2\ldots\aaqs_{n+1}\phi$. By Club Determinacy of $T^*$,
$\aaq\bt(\aaqs_1\psi \vee \aaqs_1\neg\psi)\in T^*$. By Induction Hypothesis, $\aaq\bt\aaqs_1(\neg\psi\leftrightarrow \aaqs_2\ldots\aaqs_{n+1}\neg\phi)\in T^*$. Hence $\aaq\bt(\aaqs_1\psi\vee\aaqs_1\ldots\aaqs_{n+1}\neg\phi)\in T^*$, as desired.
\end{proof}


\begin{definition}\label{ele}Suppose $\bj^T_\alpha=(J^T_\alpha,\in,T,T^*,(P)_\xi)$ is an aa-premouse  and $\bj^S_\beta=(J^S_\beta,\in,S,S^*,(P')_{\xi'})$  is an aa-premouse with $\xi\le\xi'$ and $\alpha\le\beta$.
 A mapping $\pi:J^T_\alpha\to J^S_\beta$ is called a \emph{{weak elementary embedding}} of $\bj^T_\alpha$ into $\bj^S_\beta$, in symbols $$\pi:\bj^T_\alpha\to\bj^S_\beta,$$ if $\pi$ is a first order elementary embedding  
 $$(J^T_\alpha,\in,T,(P)_\xi)\to (J^S_\beta,\in,S,(P')_{\xi'})\rest \tau_\xi^-$$ and for all $ \phi(\bx)\in \LL(\aaq)$ in the vocabulary $\tau^-_\xi$ and all $\ba\in J^T_\alpha$, $$ \phi(\ba)\in T^*\iff \phi(\pi(\ba))\in S^*.$$ 
\end{definition}

\begin{example}\label{ex2} The {canonical} example of an aa-premouse  is 
$$\mn=(J'_\alpha,\in,\Tr\rest\alpha,\Tr_\alpha,(P)_0),$$
where $\Tr$ and $\tr_\alpha$ are as in Definition~\ref{definJ} and $(P)_0$ is the empty sequence. Note that $\mn\in \CAA$. 
We obtain other examples of aa-premice by taking elementary substructures  
of $\mn$. Since $\mn\in \CAA$, we can take such also inside $\CAA$.
\end{example}


\subsubsection{The aa-ultrapower}

We define now what we call the \emph{aa-ultrapower} $(M,E,S,S^*)$ of an aa-premouse $\bj^T_\alpha$. We do not use an ultrafilter for the construction, but rather the family $\F$ of $\LL(\aaq)$-definable sets which contain (in $V$) a club of countable subsets of $J^T_\alpha$. Since we assume Club Determinacy,
this family behaves sufficiently like an ultrafilter. Thus, intuitively we define $$M=_{\mbox{\tiny def}}
(\bj^T_\alpha)^{\Pw_{\omega_1}(J^T_\alpha)}/\F,$$ where ${\Pw_{\omega_1}(J^T_\alpha)}$ is computed in $V$.
However, in the end, we cannot define $M$ in this way, at least if we want to build $M$ inside $\CAA$. We certainly cannot count on $\Pw_{\omega_1}(J^T_\alpha)$ being in $\CAA$, even though $J^T_\alpha\in \CAA$, and even though we can define sets in $\CAA$ by reference to clubs in $\Pw_{\omega_1}(J^T_\alpha)$.  

In order to prove  $\CH$ in $\CAA$, we want to build the ultrapower $M$ in $\CAA$ and therefore we modify the usual ultraproduct construction in a special way. Instead of defining $M$ as the set of equivalence classes of definable functions 
$f:\Pw_{\omega_1}(J^T_\alpha)\to J^T_\alpha$, 
we define $M$ as the set of equivalence classes of $\LL(\aaq)$-formulas $\phi(s,x)$ that define functions $f:\Pw_{\omega_1}(J^T_\alpha)\to J^T_\alpha$.

Let us now go into the details. 

\begin{definition}\label{product}
Suppose $(J^T_\alpha,\in,T,T^*,(P)_\xi)$ is an aa-premouse. 
\begin{enumerate}
\item Let $M'$ be the set of all $\phi(s,x,{\ba})$ in $\LL(\aaq)$  in the vocabulary $\tau^-_\xi$, where   ${\ba}\in J^T_\alpha$ and  $\aaqs\exists x\phi(s,x,{\ba})\in T^*$.  
\item Define in $M'$:
$$\phi(s,x,{\ba})\sim \phi'(s,x,{\ba'})\iff \aaqs(f_{\phi(s,x,{\ba})}(s)=f_{\phi'(s,x,{\ba'})}(s))\in T^*.$$
\end{enumerate}
\end{definition}

Note that $\sim$ is an equivalence relation in $M'$. Moreover,
if (1) $\rR\in\tau^-_\xi$, 
(2) the sentence $\aaqs \rR(f_{\phi_1(s,x,{\ba_1})}(s),\ldots, f_{\phi_n(s,x,{\ba_n})}(s))$ is in $T^*$, and (3) $\phi_i(s,x,{\ba_i})\sim \phi'_i(s,x,{\ba'_i})$ for $i=1,\ldots,n$, then we may easily conclude
that $$\aaqs \rR(f_{\phi'_1(s,x,{\ba'_1})}(s),\ldots, f_{\phi'_n(s,x,{\ba'_n})}(s))$$ is in $T^*$.
%
%





\begin{definition}[aa-ultrapower]\label{newpred}
The \emph{aa-ultrapower} of an aa-premouse $$(J^T_\alpha,\in,T,T^*,(P)_\xi),$$  denoted $\ult(J^T_\alpha,\in,T,T^*,(P)_\xi)$, is the $\tau_{\xi+1}$-structure $$\bM=(M,E,S,S^*,(P')_{\xi+1}),$$ where
\begin{enumerate}

\item $M$ is the set  of equivalence classes 
$[\phi(s,x,{\ba})]$ of $\sim$ on $M'$.

\item $[\phi_1(s,x,{\ba_1})]E[\phi_2(s,x,{\ba_2})]\mbox{ iff }
\aaqs \rR_\in(f_{\phi_1(s,x,{\ba_1})}(s),f_{\phi_2(s,x,{\ba_2})}(s))\in T^*.$

\item $([\phi_1(s,x,{\ba_1})],[\phi_2(s,x,{\ba_2})])\in S\mbox{ iff }
\aaqs \rR_T(f_{\phi_1(s,x,{\ba_1})}(s),f_{\phi_2(s,x,{\ba_2})}(s))$ $\in T^*.$

\item  $S^*$ consists of
$\psi(\rP_{\xi},[{\phi_1(s,x,\ba)}],\ldots,[{\phi_n(s,x,\ba)}]),$ where
$\psi(s,x_1,\ldots,x_n)$ is a $\tau_\xi^-$-formula of $\LL(\aaq)$ such that
$$ \aaqs\psi(s,f_{\phi_1(s,x,\ba)}(s),\ldots,f_{\phi_n(s,x,\ba)}(s))\in T^*.$$

\item $[\phi(s,x,{\ba})]\in P'_\eta\mbox{ iff }
\aaqs \rP_\eta(f_{\phi(s,x,{\ba})}(s))\in T^*,$
for $\eta<\xi.$

\item $P'_{\xi}=\{j(a):a\in J^T_\alpha\}$, where  $j:J^T_\alpha\to M$ is the {\em canonical embedding} $j(a)=[x=a]$. 

\end{enumerate}

\end{definition}

Note that $\ult(J^T_\alpha,\in,T,T^*,(P)_\xi))$ has  one unary predicate more in its vocabulary than  $(J^T_\alpha,{\in,}T,T^*,(P)_\xi)$ itself, namely $P_{\xi}$. Thus the aa-ultrapower extends the model but also expands the vocabulary.
These new unary predicates play a crucial role when we apply  aa-ultrapowers.

We will show that $\ult(J^T_\alpha,\in,T,T^*,(P)_\xi))$, if well-founded, is an aa-premouse. To that end we need a sequence of lemmas.

\begin{lemma}\label{consistency}
 $S^*$ is a  complete and consistent $\LL(\aaq)$-theory  in the vocabulary $\tau_{\xi+1}^-$  with constants $c_a$ for $a\in M$. 
\end{lemma}

\begin{proof}As to consistency,
suppose $\psi_i(\rP_\xi,[{\phi_1(t,x,\ba)}],\ldots,[{\phi_n(t,x,\ba)}])$, where $i=1,\ldots,m$, is a finite set of sentences in $S^*$ such that 
\begin{equation}\label{dops}
\bigwedge_{i=1}^m\psi_i(\rP_\xi,,[{\phi_1(s,x,\ba)}],\ldots,[{\phi_n(s,x,\ba)}])\vdash\bot.
\end{equation}
 By the definition of $S^*$, for $i=1,\ldots,m$ $$ \bigwedge_{i=1}^m\aaqs\psi_i(s,f_{\phi_1(s,x,\ba)}(s),\ldots,f_{\phi_n(s,x,\ba)}(s))\in T^*,$$ whence
$$ \aaqs\bigwedge_{i=1}^m\psi_i(s,f_{\phi_1(s,x,\ba)}(s),\ldots,f_{\phi_n(s,x,\ba)}(s))\in T^*.$$ It can be shown by induction on $\LL(\aaq)$-proofs that (\ref{dops}) implies
$$\aaq  s\bigwedge_{i=1}^m\psi_i(s,f_{\phi_1(s,x,\ba)}(s),\ldots,f_{\phi_n(s,x,\ba)}(s))\vdash\aaqs\bot,$$ whence $\aaqs\bot\in T^*$, contrary to the consistency of $T^*$. 

Completeness follows from Club Determinacy.
\end{proof}

\begin{lemma}\label{countable}
$\aaqs\forall x(\rP_{\xi}(x)\to s(x))\in S^*$ i.e. $\rP_{\xi}$ is ``countable" in the sense of $S^*$.
\end{lemma}

\begin{proof} Clearly, $\aaqt
\aaqs\forall x(t(x)\to s(x))\in T^*$. Hence by Definition \ref{newpred},  we have
$\aaqs\forall x(\rP_{\xi}(x)\to s(x))\in S^*.$
\end{proof}

We prove an analogue of \L o\'s Lemma first for first order formulas only. In fact, we do not have proper control of truth of formulas of stationary logic in the potentially countable model $\bM$.

\begin{lemma}[\L o\'s Lemma for first order formulas]\label{los1} 
Suppose $(J^T_\alpha,\in,T,T^*,(P)_\xi)$ is an aa-premouse  and $$\bM=\ult(J^T_\alpha,\in,T,T^*,(P)_\xi)=(M,E,S,S^*,(P')_{\xi+1}).$$ The following  are equivalent for first order formulas $\phi(s,x_1,\ldots,x_n)$ in the vocabulary $\tau^-_{\xi+1}$:
\begin{description}
\item [(1)] $\bM\models\phi(\rP_{\xi},[\phi_1(s,x,{\ba_1})],\ldots, [\phi_n(s,x,{\ba_n})])$.
\item [(2)] $\aaqs\phi(s,f_{\phi_1(s,x,{\ba_1})}(s),\ldots, f_{\phi_n(s,x,{\ba_n})}(s))\in T^*$
\end{description}

\end{lemma}

\begin{proof}
This is proved by induction on $\phi(s,x_1,\ldots,x_n)$.  The case of the atomic formula $\rP_{\xi}(x)$ follows from Definition~\ref{newpred} (6) and axiom (A2). As the formula $\phi(s,x_1,\ldots,x_n)$ is first order, the only other case that requires a proof is the case of the existential quantifier.
Assume $\phi(s,x_1,\ldots,x_n)$ is $\exists y\psi(y,s,x_1,\ldots,x_n)$. Then: 
$$\begin{array}{l}
\bM\models\phi(\rP_{\xi},[\phi_1(s,x,{\ba_1})],\ldots, [\phi_n(s,x,{\ba_n})])\Rightarrow\\
\bM\models \exists y\psi(y,\rP_{\xi},[\phi_1(s,x,{\ba_1})],\ldots, [\phi_n(s,x,{\ba_n})])\Rightarrow\\
\bM\models\psi([\theta(s,x,\bb)],\rP_{\xi}, [\phi_1(s,x,{\ba_1})],\ldots, [\phi_n(s,x,{\ba_n})])\mbox{ for some $\theta(s,x,\bb)$}\Rightarrow\\
\aaqs \psi(f_{\theta(s,x,\bb)}(s), s,f_{\phi_1(s,x,{\ba_1})}(s)
,\ldots, f_{\phi_n(s,x,{\ba_n})}(s))\in T^*\mbox{ for some $\theta(s,x,\bb)$}\Rightarrow\\
\aaqs \exists y\psi(y, s,f_{\phi_1(s,x,{\ba_1})}(s)
,\ldots, f_{\phi_n(s,x,{\ba_n})}(s))\in T^*\Rightarrow\\
\aaqs\phi(s,f_{\phi_1(s,x,{\ba_1})}(s),\ldots, f_{\phi_n(s,x,{\ba_n})}(\bs))\in T^*\\
\end{array}$$ and, on the other hand, letting $\theta(s,y,\ba_1,\ldots,\ba_n)$ be the formula 
$$\psi(y, s,f_{\phi_1(s,x,{\ba_1})}(s),\ldots, f_{\phi_n(s,x,{\ba_n})}(s)),$$ we obtain
$$\begin{array}{l}
\aaqs\phi(s,f_{\phi_1(s,x,{\ba_1})}(s),\ldots, f_{\phi_n(s,x,{\ba_n})}(s))\in T^*\Rightarrow\\
\aaqs \exists y\psi(y, s,f_{\phi_1(s,x,{\ba_1})}(s),\ldots, f_{\phi_n(s,x,{\ba_n})}(s))\in T^*\Rightarrow\\
\aaqs \psi(f_{\theta(s,y,\ba_1,\ldots,\ba_n)}(s), s,f_{\phi_1(s,x,{\ba_1})}(s),\ldots, f_{\phi_n(s,x,{\ba_n})}(s))\in T^*\Rightarrow\\
\bM\models\psi([\theta(s,y,\ba_1,\ldots,\ba_n],\rP_{\xi},[\phi_1(s,x,{\ba_1})],\ldots, [\phi_n(s,x,{\ba_n})])\Rightarrow\\
\bM\models \exists y\psi(y,\rP_{\xi},[\phi_1(s,x,{\ba_1})],\ldots, [\phi_n(s,x,{\ba_n})])\Rightarrow\\
\bM\models\phi(\rP_{\xi},[\phi_1(s,x,{\ba_1})],\ldots, [\phi_n(s,x,{\ba_n})])\end{array}$$
\end{proof}

We can now show that the canonical embedding is a first order elementary embedding:

\begin{lemma}\label{firstorder}
If $\phi$ is a first order formula in $\tau^-_0$, then the following conditions are equivalent:
\begin{description}
\item [(1)] $\bj^T_\alpha\models\phi(\ba)$.
\item [(2)] $\bM\models\phi(j(\ba))$.
\end{description}
\end{lemma}

\begin{proof}
If (1) holds, then $\phi(\ba)\in T^*$, whence $\aaqs \phi(\ba)\in T^*$, and further $\bM\models \phi(j(\ba))$ by Lemma~\ref{los1}.
If (1) fails, then $\neg\phi(\ba)\in T^*$, whence $\aaqs \neg\phi(\ba)\in T^*$, and further $\bM\not\models \phi(i(\ba))$, by the consistency of $T^*$.
\end{proof}

\begin{lemma}\label{2777}\label{preserved}
$j[P_\eta]= P'_\eta$ for any predicate $\rP_\eta\in \tau^-_\xi$.
\end{lemma} 

\begin{proof}
The inclusion $j[P_\eta]\subseteq P'_\eta$ is trivial. The opposite direction follows from Axiom (A5) and Definition~\ref{prem} (3).
\end{proof}

It is a consequence of the above Lemma that the following  are equivalent for first order formulas $\phi(x_1,\ldots,x_n)$ in the vocabulary $\tau^-_0$:
\begin{description}
\item [(1)] $\phi(\ba)\in T_\beta$.
\item [(2)] $\phi(j(\ba))\in S_{j(\beta)}$.
\end{description}

Now we prove an analogue of \L o\'s Lemma for $\LL(\aaq)$. Unlike in Lemma~\ref{los1} we do not talk about truth in the aa-ultrapower but only about membership in the theories $T^*$ and $S^*$.

\begin{lemma}[\L o\'s Lemma for $\LL(\aaq)$]\label{gh} Suppose $(J^T_\alpha,\in,T,T^*,(P)_\xi)$ is an aa-premouse  and $$(M,E,S,S^*,(P')_{\xi+1})=\ult(J^T_\alpha,\in,T,T^*,(P)_\xi).$$ The following  are equivalent for $\LL(\aaq)$-formulas $\phi(\rP_\xi,x_1,\ldots,x_n)$ in the vocabulary $\tau^-_{\xi+1}$:
\begin{description}
\item [(1)] $\phi(\rP_{\xi},[\phi_1(s,x,{\ba_1})],\ldots, [\phi_n(s,x,{\ba_n})])\in S^*$.
\item [(2)] $\aaqs\phi(s,f_{\phi_1(s,x,{\ba_1})}(s),\ldots, f_{\phi_n(s,x,{\ba_n})}(s))\in T^*$.
\end{description}

\end{lemma}

\begin{proof}The implication (2)$\to$(1) is built into the definition of $S^*$. If (2) fails, then 
$\neg\aaqs\phi(s,f_{\phi_1(s,x,{\ba_1})}(s),\ldots, f_{\phi_n(s,x,{\ba_n})}(s))\in T^*$, by the completeness of $T^*$
in the vocabulary $\tau_\xi^-$. By Club Determinacy of $T^*$, $$\aaqs\neg\phi(s,f_{\phi_1(s,x,{\ba_1})}(s),\ldots, f_{\phi_n(s,x,{\ba_n})}(s))\in T^*,$$ whence $\neg\phi(\rP_{\xi},[\phi_1(s,x,{\ba_1})],\ldots, [\phi_n(s,x,{\ba_n})])\in S^*$.  Now (1) fails because of the consistency of $S^*$ (Lemma~\ref{consistency}).\end{proof}

It is a consequence of the above Lemma that the following  are equivalent for $\LL(\aaq)$-formulas $\phi(x_1,\ldots,x_n)$ in the vocabulary $\tau^-_\xi$:
\begin{description}
\item [(1)] $\phi([\phi_1(s,x,{\ba_1})],\ldots, [\phi_n(s,x,{\ba_n})])\in S^*$.
\item [(2)] $\aaqs\phi(f_{\phi_1(s,x,{\ba_1})}(s),\ldots, f_{\phi_n(s,x,{\ba_n})}(s))\in T^*$.
\end{description}
In particular,  $\phi(\ba)\in T^*$ iff $\phi(j(\ba))\in S^*$ for $\LL(\aaq)$-sentences $\phi$ in the 
vocabulary $\tau^-_\xi$.

\begin{lemma}\label{wf}The aa-ultrapower
$(M,E,S,S^*,(P')_{\xi+1})$, if well-founded, collapses to an aa-premouse 
$(J^{\bar{T}}_\beta,\in,\bar{T},\bar{T}^*,(\bP)_{\xi+1})$ with vocabulary $\tau_{\xi+1}$. 
The canonical mapping $j$, composed with the collapse function $\pi:(M,E)\to(J^{\bar{T}}_\beta,\in)$, is a weak elementary embedding\footnote{In the sense of Definition~\ref{ele}.} $$i:(J^T_\alpha,\in,T,T^*,(P)_{\xi})\to(J^{\bar{T}}_\beta,\in,\bar{T},\bar{T}^*,(\bP)_{\xi+1}).$$ 
\end{lemma}

\begin{proof}It follows from Lemma~\ref{firstorder}, that the aa-ultrapower
$(M,E,S,S^*,(P')_{\xi+1})$, if well-founded, collapses to  a structure of the type $(J^{\bar{T}}_\beta,\in,\bar{T},\bar{T}^*,(\bar{P})_{\xi+1})$ with vocabulary $\tau_{\xi+1}$. We  only have to show that the conditions of Definition~\ref{prem} are satisfied by this structure.

(1) It follows from Lemma~\ref{firstorder} and an argument similar to the proof of Lemma~\ref{consistency}, that $\bT\subseteq \beta\times \LL(\aaq)$, and  for all 
$\gamma\in\beta\cap\Lim$, the set $\bT_\gamma=\{\phi(\ba):(\gamma,\phi(\ba))\in \bT,$ and $\ba\in J^{\bar{T}}_\gamma\}$ is a complete consistent $\LL(\aaq)$-theory in the vocabulary $\tau^-_{0}$ extending the first order theory of $(J^{\bar{T}}_\gamma,\in,\bT\rest \gamma)$, where we allow constants $c_a$ for $a\in J^{\bar{T}}_\gamma$.

(2) The completeness and consistency of $S^*$ was already proved in Lemma~\ref{consistency}.
By  Lemma~\ref{los1}, $S^*$ extends the first order theory of 
$(J^{\bT}_\beta,\in,\bT,(\bP)_{\xi+1})$.



(3) By Lemma \ref{2777} the sequence $(P')_{\xi}$ is continuously increasing. Moreover, Definition~\ref{prem} (3) implies $P'_\eta\subseteq P'_{\xi}$ for all $\eta< \xi$.

(4) We show, that if $\exists x\phi(x,\rP_{\xi},\pi(\ba))\in \bT^*$, where $\ba\in M$, then there is $b\in M$ such that $\phi(\pi(b),\rP_{\xi},\pi(\ba))\in \bT^*$. Assume $\ba$ has length one, for simplicity. Accordingly, suppose $\exists x\phi(x,\rP_{\xi},\pi([\psi(s,x,\bc)]))\in \bT^*$ for some $\psi(s,x,\bc)$ with $\bc\in J'_\alpha$ and $\aaqs\exists x
\psi(s,x,\bc)\in T^*$. This implies $\exists x\phi(x,\rP_{\xi},[\psi(s,x,\bc)])\in S^*$. Hence $\aaqs\exists x
\phi(x,s,f_{\psi(s,x,\bc)}(s))\in T^*$. Therefore $$\aaqs\phi(f_{\phi(x,s,f_{\psi(s,x,\bc)}(s))}(s),s,f_{\psi(s,x,\bc)}(s))\in T^*.$$ Hence
$$\phi([\phi(x,s,f_{\psi(s,x,\bc)}(s))],\rP_{\xi},[\psi(s,x,\bc)])\in S^*.$$ We let $b=[\phi(x,s,f_{\psi(s,x,\bc)}(s))]$ and the claim is proved.

(5) If $$\begin{array}{l}
\aaqss\exists x\phi(\rP_{\xi},x,\vs,\pi(\ba))\to\\
\aaqss\exists x(\phi(\rP_{\xi},x,\vs,\pi(\ba))\wedge\forall y\prec x\neg\phi(\rP_{\xi},y,\vs,\pi(\ba)))
\end{array}$$ is not in $\bT^*$, then $\aaqss\exists x\phi(\rP_{\xi},x,\vs,\pi(\ba))\in \bT^*$ and $$\neg\aaqss\exists x(\phi(\rP_{\xi},x,\vs,\pi(\ba))\wedge\forall y\prec x\neg\phi(\rP_{\xi},y,\vs,\pi(\ba)))$$ is in $\bT^*$. Hence $\aaq t\aaqss\exists x\phi(t,x,\vs,\ba)\in T^*$ and $$\aaq t\neg\aaqss\exists x(\phi(t,x,\vs,\ba)\wedge\forall y\prec x\neg\phi(t,y,\vs,\ba))\in T^*.$$ By Club Determinacy  of $T^*$, $$\aaq t\aaqss\neg\exists x(\phi(t,x,\vs,\ba)\wedge\forall y\prec x\neg\phi(t,y,\vs,\ba))\in T^*,$$ a contradiction with the assumption that $T^*$ satisfies condition (5) of Definition~\ref{prem}.

(6) We show that the Club Determinacy schema (\ref{cds}) is contained in $\bT^*$. Suppose $\psi(\rP_{\xi},x_1,\ldots,x_n)$ is an $\LL(\aaq)$-formula in the vocabulary $\tau^-_\xi$ and 
$$[\phi_1(u,x,\ba)],\ldots,[\phi_n(u,x,\ba)]\in M.$$ By the Club Determinacy  
of $T^*$ 
$$\begin{array}{l}
\aaq u\aaq\bt(\aaqs\psi(u,\bt,f_{\phi_1(u,x,\ba)}(u),\ldots,f_{\phi_n(u,x,\ba)}(u))\vee\\
\hspace{1cm}\aaqs\neg\psi(u,\bt,f_{\phi_1(u,x,\ba)}(u),\ldots,f_{\phi_n(u,x,\ba)}(u)))\in T^*.
\end{array}$$
 Hence,
$$\begin{array}{l}
\aaq\bt(\aaqs\psi(\rP_{\xi},\bt,[\phi_1(u,x,\ba)],\ldots,[\phi_n(u,x,\ba)])\vee\\
\hspace{1cm}\aaqs\neg\psi(\rP_{\xi},\bt,[\phi_1(u,x,\ba)],\ldots,[\phi_n(u,x,\ba)]))\in S^*.
\end{array}$$

(7)  It is a consequence of Lemma~\ref{gh} and  $\{\aaqs\exists x\neg x,\aaqs(\omega\subseteq s)\}\subseteq T^*$, that $\aaqs\exists x\neg x\in s$ and $\aaqs(\omega\subseteq s)$ are in $\bT^*$.
 
(8) This condition can be expressed, in view of Definition~\ref{prem} (4) as the membership of a universal sentence in $T^*$. The vocabulary of this sentence is $\tau_0^-$, and so it is also an element of $S^*$. 


(9) This is obvious.

Finally, $j$ is weak elementary by Lemma~\ref{los1} and because
$$\phi(a_1,\ldots,a_n)\in T^*\iff \aaqs \phi(f_{x=a_1}(s),\ldots,f_{x=a_n}(s))\in T^*$$
$$\iff \phi([x=a_1],\ldots,[x=a_n])\in S^*\iff \phi(j(a_1),\ldots,j(a_n))\in S^*.$$

\end{proof}

\begin{lemma}\label{proper}
$j[J^T_\alpha]\ne M$.
\end{lemma}

\begin{proof}
We consider $[\neg x\in s]\in M$. Suppose $[\neg x\in s]=i(a)$ for some $a\in J^T_\alpha$, i.e. $[ \neg x\in s]=[x=a]$. Then $\aaqs (f_{\neg x\in s}(s)=f_{x=a}(s))\in T$, whence $\aaqs (a\not\in s)\in T$. But by Axiom (A2)  of stationary logic $\aaqs (a\in s)\in T$, a contradiction.
\end{proof}

\begin{lemma}
If there is $\gamma<\alpha$ such that $\neg\aaqs(\gamma\subseteq s)\in T^*$ and $\gamma$ is the least such, then $\gamma$ is the critical point of $j$.
\end{lemma}

\begin{proof}Suppose $\alpha<\gamma$.
If $[\phi(s,x,\ba)]<j(\alpha)$, then Axiom A5 (Fodor's Lemma) implies that there is $\delta<\gamma$ such that $[\phi(s,x,\ba)]=\delta$. Hence $i(\alpha)=\alpha$. On the other hand $[x\in \gamma\wedge x\notin s]$ demonstrates that $i(\gamma)>\gamma$.
\end{proof}

We shall now prove an important Lemma which, among other things, shows that for the kind of aa-premice that we are mainly interested in, namely those arising from Example~\ref{ex2}, the aa-ultrapower of a well-founded aa-premouse is well-founded.

\begin{definition} A $\tau_\xi$-structure  $(J'_\beta,\in,{\tr}\rest\beta, {\tr}',(P)_\xi)$ is called \emph{aa-like with respect to (w.r.t.) $M$}, $M\subseteq J'_\beta$, if ${\tr}'$ is a complete consistent $\LL(\aaq)$-theory in the vocabulary $\tau^-_\xi$ with parameters from $J'_\beta$,  ${\tr}'\restriction\tau^-_0= \tr_\beta$ and $$\phi(\rP_{\gamma_1},\ldots \rP_{\gamma_n},\ba)\in {\tr}'\Rightarrow(J'_\beta,\in,{\tr}\rest\beta, (P)_\xi)\models\aaqs
 \phi(\rP_{\gamma_1},\ldots \rP_{\gamma_{n-1}},s,\ba)$$ for all sentences $\phi(\rP_{\gamma_1},\ldots \rP_{\gamma_n},\ba)\in\LL(\aaq)$  in the vocabulary $\tau^-_0\cup\{\rP_{\gamma_1},\ldots,\rP_{\gamma_n}\}$, where $\gamma_1<\ldots<\gamma_n<\xi$, and for all  $\ba\in M$.
\end{definition}

\begin{lemma}\label{wff}  Suppose $(J^T_\alpha,\in,T,T^*,(P)_\xi)$ is a countable  aa-premouse and 

\begin{equation}\label{weak}
\begin{array}{l}\pi:{(J^T_\alpha,\in,T,T^*,(P)_\xi)}\to
N=(J'_\beta,\in,{\tr}\rest\beta,{\tr}',(P')_\xi)
\end{array}
\end{equation}
 is a weak elementary embedding such that $N$ is aa-like w.r.t. $\ran(\pi)$. There are
$P'_{\xi}\subseteq J'_\beta$ and  a weak elementary
$$\pi^*:\ult(J^T_\alpha,\in,T,T^*,(P)_\xi)\to \bar{N}=(J'_\beta,\in,{\tr}\rest\beta,{\tr}'',(P')_{\xi+1})$$ 
such that   $\pi^*(i(a))=\pi(a)$ for all $a\in J^T_\alpha$, and  $\bar{N}$ is aa-like w.r.t.
$\ran(\pi^*)$.
\end{lemma}

\begin{proof}
Suppose $[\phi(s,x,{\ba})]\in M$. Then $\aaqs\exists x\phi(s,x,{\ba})\in T^*$, whence we have $\aaqs\exists x\phi(s,x,\pi({\ba}))\in {\tr}'$ and hence by aa-likeness, $N\models\aaqs\exists x\phi(s,x,\pi({\ba}))$. Let $C_{{\ba},\phi}$ be a club of countable subsets $s$ of $J'_{\beta}$ such that $N\models\exists x\phi(s,x,\pi({\ba}))$. Let $Q$ be the intersection of the countably many $C_{{\ba},\phi}$ where ${\ba}\in J^T_\alpha$, $\phi\in\LL(\aaq)$, and $\aaqs\exists x\phi(s,x,\pi({\ba}))\in {\tr}'$. Let us fix $s^*\in Q$. Note that $s^*$ need not be in $\CAA$. Now for all ${\ba}\in J^T_\alpha$ and $\phi\in\LL(\aaq)$ such that $\aaqs\exists x\phi(s,x,\pi({\ba}))\in {\tr}'$ there is a $\prec$-least $z_{\ba,\phi}\in N$ such that $N\models\phi(s^*,z_{\ba,\phi},\pi({\ba}))$, i.e. $f_{\phi(s,x,{\pi(\ba)})}(s^*)=z_{\ba,\phi}$.  We let $\pi^*([\phi(s,x,{\ba})])=z_{\ba,\phi}$ and $P'_{\xi}=s^*$. 
Obviously, $\pi^*(j(a))=\pi(a)$ for all $a\in J^T_\alpha$.
Let ${\tr}''$ be a complete extension  of ${\tr'}$ together with the  $\tau^-_{\xi+1}$-sentences $$\psi(\rP_{\xi},\pi^*([\phi_1(s,x,\ba)]),\ldots,\pi^*([\phi_m(s,x,\ba)]))$$ such that 
$$\aaqs \psi(s,f_{\phi_1(s,x,\pi(\ba))}(s),\ldots, f_{\phi_n(s,x,\pi(\ba))}(s))\in {\tr}'.$$

Clearly now  $\bar{N}$ is aa-like w.r.t. $\ran(\pi^*)$. 

Now we prove that $\pi^*$ is elementary. Because of  Club Determinacy of $T^*$ it suffices to prove one direction.
Suppose  
$$(M,E,S,(P)_{\xi+1})\models \psi(\rP_{\xi},[\phi_1(s,x,\ba)],\ldots,[\phi_m(s,x,\ba)]),$$ where  $\psi(s,x_1,\ldots,x_m)$ is a first order $\tau^-_\xi$-formula. By Lemma~\ref{los1}
$$\aaqs \psi(s,f_{\phi_1(s,x,\ba)}(s),\ldots, f_{\phi_n(s,x,\ba)}(s))\in T^*,$$ 
whence 
by (\ref{weak}) $$\aaqs \psi(s,f_{\phi_1(s,x,\pi(\ba))}(s),\ldots, f_{\phi_n(s,x,\pi(\ba))}(s))\in {\tr}'.$$
Hence
$$\psi(\rP_{\xi},\pi^*([\phi_1(s,x,\pi(\ba))]),\ldots,\pi^*([\phi_m(s,x,\pi(\ba))]))\in {\tr}''.$$


Suppose next $ \psi(\rP_\xi,\bx)\in \LL(\aaq)$ in the vocabulary $\tau^-_{\xi+1}$ and   $$\psi(\rP_\xi,[\phi_1(s,x,\ba)],\ldots,[\phi_m(s,x,\ba)])\in S^*.$$ By the definition of $S^*$ in Definition~\ref{newpred}, condition 4,
$$\aaqs \psi(s,f_{\phi_1(s,x,\ba)}(s),\ldots, f_{\phi_n(s,x,\ba)}(s))\in T^*,$$ whence by (\ref{weak})
$$\aaqs \psi(s,f_{\phi_1(s,x,\pi(\ba))}(s),\ldots, f_{\phi_n(s,x,\pi(\ba))}(s))\in{\tr}'.$$
Hence
$$\psi(\rP_{\xi},\pi^*([\phi_1(s,x,\pi(\ba))]),\ldots,\pi^*([\phi_m(s,x,\pi(\ba))]))\in{\tr}''.$$ 
%
%
\end{proof}

We can iterate the aa-ultrapower construction and this will be henceforth our main tool:

\begin{definition}\label{system}
We define a directed system 
\begin{equation}\label{aaiteration}
\langle(M_\beta,E_\beta,T_\beta,T^*_\beta,
(P^\beta)_\beta),j_{\alpha,\beta}:\alpha<\beta\le\omega_1\rangle
\end{equation} of  structures,\footnote{Each structure $(M_\beta,E_\beta)$ will be shown to be well-founded, when we actually use 
this construction, so then these structures are  aa-premice, up to isomorphism.}
called an  \emph{aa-iteration starting from} $(M_0,\in,T_0,T^*_0,(P^0)_0)$, as follows:
\begin{description}
\item[(1)] The vocabulary of the structure $(M_\beta,E_\beta,T_\beta,T^*_\beta,(P^\beta)_\beta)$ is $\tau_\beta$.

\item [(2)]We have a commuting system of weak elementary embeddings
$$j_{\alpha\beta}:(M_\alpha,E_\alpha,T_\alpha,T^*_\alpha,(P^\alpha)_\alpha)\to (M_\beta,E_\beta,T_\beta,T^*_\beta,(P^\beta)_\beta)\rest\tau_\alpha.$$ 

\item[(3)]   $(M_0,\in,T_0,T^*_0,(P^0)_0)$ is  a countable aa-premouse with vocabulary $\tau_0$. 

\item[(4)] At successor stages we let 
$$\begin{array}{l}
(M_{\alpha+1},E_{\alpha+1},T_{\alpha+1},T^*_{\alpha+1},
(P^{\alpha+1})_{\alpha+1})=\ult(M_{\alpha},E_{\alpha},T_{\alpha},T^*_{\alpha},(P^\alpha)_\alpha).
\end{array}$$ 
 The mapping $j_{\alpha,\alpha+1}$ is the canonical elementary mapping of an aa-premouse into its aa-ultrapower. 

\item[(5)] At limit stages $(M_{\alpha},E_{\alpha},T_{\alpha},T^*_{\alpha},
(P^{\alpha})_{\alpha})$ is the direct limit of the directed system $$\langle(M_\beta,E_\beta,T_\beta,T^*_\beta,(P^{\beta})_\beta),j_{\gamma,\beta}:\gamma<\beta<\alpha, \gamma,\beta\rangle,$$ i.e. $(M_{\alpha},E_{\alpha},T_{\alpha},T^*_{\alpha},
(P^{\alpha})_{\eta})$ is the direct limit of 
$$\langle(M_\beta,E_\beta,T_\beta,T^*_\beta,(P^{\beta})_\eta),j_{\gamma,\beta}:\eta\le\gamma<\beta<\alpha, \gamma\rangle,$$ for $\eta<\alpha$.

\end{description}
\end{definition}

\begin{lemma}\label{69} Suppose $$\langle(M_\beta,E_\beta,T_\beta,T^*_\beta,
(P^\beta)_\beta),j_{\alpha,\beta}:\alpha<\beta\le\omega_1,\alpha\rangle$$ is as in Definition~\ref{system}. Let  $\delta\le\omega_1$, $\delta\in\Lim$. Suppose each of the models $(M_\beta,E_\beta,T_\beta,T^*_\beta,
(P^\beta)_\beta)$, $\beta<\delta$, is isomorphic to an aa-premouse.
Then, if well-founded, $(M_{\delta},E_{\delta},T_{\delta},T^*_{\delta},
(P^{\delta})_{\delta})$ collapses to an aa-premouse.
The canonical mappings $i_{\nu,{\delta}}$ are   elementary embeddings $$j_{\nu,{\delta}}:(M_\nu,\in,T_\nu,T_\nu^*,(P^\nu)_\nu)\to (M_{\delta},E_{\delta},T_{\delta},T^*_{\delta},
(P^{\delta})_{\delta})\rest\tau_\nu.$$

\end{lemma}

\begin{proof}
This is like the proof of Lemma~\ref{wf}.
\end{proof}

We now extend the important Lemma~\ref{wff} from single aa-ultrapowers to the context of iterated aa-ultrapowers:

\begin{lemma}\label{69a} Let $N=(J'_\zeta,\in,{\tr}\rest\zeta)$, where $\zeta$ is any limit ordinal.  Let $$\la(M_\beta,E_{\beta},T_\beta,T^*_\beta,(P^\beta)_\beta),j_{\beta,\gamma}:\beta\le\gamma\le\omega_1\ra,$$  be an aa-iteration. Let  $\delta$ be a limit ordinal $\le\omega_1$.
 Suppose for all $\beta<{\delta}$ there is a weak elementary $$\sigma_\beta:(M_\beta,E_{\beta},T_\beta,T^*_\beta,(P^\beta)_\beta)\to N_{\beta},$$ where    $N_{\beta}$ is an expansion of  $N$, aa-like w.r.t.
 $\bigcup_{\gamma<\beta}\ran(\sigma_\gamma)$, to a $\tau_\beta$-structure   such that $N_{\beta}=N_\gamma\restriction\tau_{{\beta}}$ whenever $\beta<\gamma\le\delta$. 
Then there is an expansion $N_{\delta}$ of $N$ to a $\tau_{\omega\delta}$-structure and an elementary $$\sigma_{\delta}:(M_{\delta},E_{\delta},T_{\delta},T^*_{\delta},(P^{\delta})_{\delta})\to N_{\delta}$$ such that $N_\beta=N_{\delta}\restriction\tau_{\beta}$ and $\sigma_\beta(x)=\sigma_{\delta}(j_{\beta,{\delta}}(x))$ for all $x\in M_\beta$ and all $\beta\in{\delta}$. Moreover, $N_\delta$ is aa-like w.r.t. $\ran(\sigma_{\delta})$.

\end{lemma}

\begin{proof}
The condition $N_\beta=N_{\omega\delta}\restriction\tau_{\beta}$ for $\beta\in{\omega\delta}$ determines a unique $\tau_{\omega\delta}$-structure $N_\delta$ apart from the interpretation of $P^{\delta}_{\delta}$. We let the interpretation of $\rP^{\delta}_{\delta}$ in $N_{\delta}$ to be the union of the interpretations of $\rP^{\delta}_{\beta}$, $\beta<\delta$, in $N_{\delta}$. For defining $\sigma_{\delta}$, let $a\in M_{\delta}$. There is $\beta<\delta$ such that 
$a=j_{\beta,\delta}(b)$ for some $b\in M_\beta$. We let $\sigma_{\delta}(a)=\sigma_{\beta}(b)$. Basic properties of directed limits guarantee that this is a coherent definition of a function and that the mapping $\sigma_{\delta}$ is an elementary embedding.
\end{proof}

By combining Lemma~\ref{wff} and Lemma~\ref{69a} we can be sure that all structures in the directed system of Definition~\ref{system} are well-founded and by Lemma~\ref{wf} collapse to aa-premice.


%

\begin{definition} We call the aa-premice $(M_\beta,E_\beta,T_\beta,T^*_\beta,(P^\beta)_\beta)$ \emph{{iterates}} of the aa-premouse $(M_0,E_0,T_0,T^*_0,(P)_0)$. An aa-premouse $(M_0,E_0,T_0,T^*_0,(P)_0)$ is an \emph{{\bf aa-mouse}} if  its  $\beta$'th iterate $(M_\beta,T_\beta,T^*_\beta,(P^\beta)_\beta)$ is well-founded for
all $\beta<\omega_1$. In this case  we say that the aa-premouse  $(M_0,E_0,T_0,T^*_0,(P)_0)$ is  \emph{{iterable}}. 
\end{definition}

Note that if the iterates $(M_\alpha,T_\alpha,T^*_\alpha,(P^\alpha)_\alpha)$, $\alpha<\omega_1$, are all well-founded, then also the iterate $(M_{\omega_1},T_{\omega_1},T^*_{\omega_1},(P^{\omega_1})_{\omega_1})$  is well-founded.

\begin{lemma}Suppose $M_0$ is countable and
$$(M_0,E_0,T_0,T^*_0,(P)_0)\preccurlyeq(J'_\alpha,\in,{\tr}\rest\alpha,{{\tr}}_\alpha,(P')_0).$$ Then each iterate $(M_\beta,E_\beta,T_\beta,T^*_\beta,(P^\beta)_\beta)$, $\beta\le\omega_1$, in the aa-iteration starting from $(M_0,E_0,T_0,T^*_0,(P)_0)$ is (isomorphic to) an aa-mouse.
\end{lemma}

\begin{proof}
We may use Lemmas~\ref{wff} and \ref{69a} inductively to build
$$\pi_\beta:(M_\beta,E_\beta,T_\beta,T^*_\beta,(P^\beta)_\beta)\to N_\beta$$
for all $\beta\le\omega_1$, where each $N_\beta$ is an aa-like w.r.t.
 $\bigcup_{\gamma<\beta}\ran(\sigma_\gamma)$ expansion of $(J'_\alpha,\in,{\tr}\rest\alpha)$, with the consequence that each $(M_\beta,E_\beta)$ is well-founded.
\end{proof}

\begin{lemma}\label{club} Let $\la(M_\beta,E_\beta,T_\beta,T^*_\beta,(P^\beta)_\beta),j_{\beta,\gamma}:\beta\le\gamma\le\omega_1\ra$ be an aa-iteration. 
Then the set  $C=\{(\rP_\alpha)^{M_{\omega_1}}: \alpha\in\omega_1\}$ is a club in $\Pw_{\omega_1}(M_{\omega_1})$.
\end{lemma}

\begin{proof}By Lemmas \ref{countable} and \ref{preserved}, and since we take direct limits at limit stages,  the sequence  $(\rP_\alpha)^{M_{\omega_1}}$, $\alpha<\omega_1$, is continuously increasing. By Lemma~\ref{proper} it is properly increasing. Suppose now $s$ is a  countable subset of $M_{\omega_1}$. There are $\alpha<\omega_1$ and a countable $s^*\subseteq M_\alpha$ such that $s=j_{\alpha\omega_1}[s^*]$. Hence   $s\subseteq (\rP_{\alpha+1})^{M_{\omega_1}}$.
\end{proof}

We can now prove that the final model $ (M_{\omega_1},E_{\omega_1},T_{\omega_1},T_{\omega_1}^*,(P^{\omega_1})_{\omega_1})$ of an iteration of aa-premice actually satisfies in the usual sense everything that the theory $T^*_{\omega_1}$ predicts:

\begin{proposition}\label{rre}
 Let $\la(M_\beta,E_\beta,T_\beta,T^*_\beta,(P^\beta)_\beta),j_{\beta,\gamma}:\beta\le\gamma\le\omega_1\ra$ be an aa-iteration of aa-mice.  Then for all formulas $\phi(\ba)$ of stationary logic in vocabulary $\tau^-_{\omega_1}$ and all $\ba\in M_{\omega_1}$:
$$\phi({\ba})\in T^*_{\omega_1}\iff (M_{\omega_1},E_{\omega_1},T_{\omega_1},(P^{\omega_1})_{\omega_1})\models \phi({\ba}).$$
\end{proposition}

\begin{proof}  
We prove the claim by induction on $\phi(\vec{x})$.  Let $\beta$ be the least $\beta$ such that 
$\ba=j_{\beta,\omega_1}(\ba^*)$ for some ${\ba^*}\subseteq M_\beta$. 
\medskip

\noindent 1. Atomic $\phi(\ba)$. If $\phi({\ba})\in T^*_{\omega_1}$, then $\phi({\ba^*})\in T^*_{\beta}$, whence $(M_{\beta},E_{\beta},T_{\beta},(P^{\beta})_{\beta})\models\phi({\ba^*})$ and $(M_{\omega_1},E_{\omega_1},T_{\omega_1},(P^{\omega_1})_{\omega_1})\models \phi({\ba})$ follows because $j_{\beta,\omega_1}$ is weakly elementary.
The converse follows from the completeness of $T^*_{\beta}$.

\noindent 2. Conjunction and negation: Trivial.

\noindent   3. Existential quantifier: Suppose $\exists x\phi(x,{\ba})\in T_{\omega_1}^*$ i.e. $\exists x\phi(x,{\ba^*})\in T^*_{\beta}$. Then by Definition~\ref{prem} condition (4) there is $b\in M_\beta$ such that 
 $\phi(b,{\ba^*})\in T^*_{\beta},$ whence  $\phi(j_{\beta,\omega_1}(b),{\ba})\in T^*_{\omega_1}.$
 By the Induction Hypothesis
 $(M_{\omega_1},E_{\omega_1},T_{\omega_1},(P^{\omega_1})_{\omega_1})\models\phi(j_{\beta,\omega_1}(b),{\ba}).$
 Hence we have $(M_{\omega_1},E_{\omega_1},T_{\omega_1},(P^{\omega_1})_{\omega_1})\models\exists x\phi(x,{\ba}).$ Conversely, if  $(M_{\omega_1},E_{\omega_1},T_{\omega_1},(P^{\omega_1})_{\omega_1})\models\exists x\phi(x,{\ba})$, then there is $\gamma\ge\beta$ and $b\in M_\gamma$ such that 
 $(M_{\gamma},E_{\gamma},T_{\gamma},(P^{\gamma})_{\gamma})\models\phi(b,
 j_{\beta,\gamma}(\ba^*))$. By Induction Hypothesis and the completeness of $T^*_{\omega_1}$ we have 
 $\exists x\phi(x,{\ba})\in T^*_{\omega_1}$.\medskip

\noindent 4. $\aaqs\phi(s,{\ba})$:  W.l.o.g. $\phi(s,{\ba})$ is in vocabulary $\tau^-_{\beta}$. Suppose  first the sentence $\aaqs\phi(s,{\ba})$ is in $T^*_{\omega_1}$. By weak elementarity of the mapping $j_{\beta,\omega_1}$, we have $\aaqs\phi({\ba^*})\in T^*_{\beta}$, and, moreover,  $\aaqs\phi(s,j_{\beta,\gamma}(\ba^*))\in T^*_{\gamma}$ for $\beta\le\gamma<\omega_1$.
Since the successor stages of the aa-iteration are aa-ultraproducts, $\phi(\rP_\gamma,{j_{\beta,\gamma+1}(\ba^*)})\in T^*_{\gamma+1}$  for $\beta\le\gamma<\omega_1$.  By Induction Hypothesis,  $$(M_{\omega_1},E_{\omega_1},T_{\omega_1},(P^{\omega_1})_{\omega_1})\models \phi(\rP_{\gamma},{\ba})$$ whenever
 $\beta\le\gamma<\omega_1$. By Lemma~\ref{club},  $$(M_{\omega_1},E_{\omega_1},T_{\omega_1},(P^{\omega_1})_{\omega_1})\models \aaqs\phi(s,{\ba}).$$
Conversely, if  $\aaqs\phi(s,{\ba})\notin T^*_{\beta}$, then $\aaqs\neg\phi(s,{\ba^*})\in T^*_{\beta}$ and we can argue as above.

%
%
\end{proof}

\subsubsection{The Continuum Hypothesis in $\CAA$}

We use the method of iterating the aa-ultrapower construction $\omega_1$ times to prove the Continuum Hypothesis and $\Diamond$ in $\CAA$. The proof is reminiscent of  Silver's proof of $\GCH$ in $L^\mu$ \cite{MR0278937}. 

 To this end, let $\la((M_\beta,\in,T_\beta,T^*_\beta,(P^\beta)_\beta),j_{\beta,\gamma}):\beta\le\gamma\le\omega_1\ra$ be as in Definition~\ref{system}.
%
%

\begin{lemma}\label{nonew}Suppose 
$$(M_0,\in,T_0,T^*_0,(P)_0)\prec (J'_{\omega\alpha},\in,{\tr}\rest\omega\alpha,{\tr}_{\omega\alpha},(P')_0),$$ where $\alpha$ is a limit ordinal and  $M_0$ is countable. Then
$M_{\omega_1}$ {does not have new reals} over those in $M_0$.
\end{lemma}
\begin{proof} Suppose $r$ is a real in $M_{\omega_1}$ and not of the form $j_{0,\omega_1}(r^*)$ for any real $r^*\in M_0$. Let $\xi<\omega_1$ such that $r=j_{\xi+1,\omega_1}(r^*)$ for some $r^*\in M_{\xi+1}$ and no such $r^*$ exists in $M_\xi$. Then $r^*=[\phi(s,x,\ba)]$ for some $\phi(s,x,\by)\in\LL(\aaq)$ in the vocabulary $\tau^-_{\xi+1}$ and some $\ba\in M_\xi$ such that $\aaqs\exists x\phi(s,x,\ba)\in T^*_\xi$. 
In particular, $M_{\omega_1}\models \aaqs\exists x\phi(s,x,j_{\xi,\omega_1}(\ba))$. 
Since $M_{\omega_1}\models ``[\phi(s,x,j_{\xi,\omega_1}(\ba))]\subseteq\omega"$, the sentence
$\exists x(x\subseteq\omega\wedge\forall n(n\in[\phi(s,x,j_{\xi,\omega_1}(\ba))]\leftrightarrow n\in x))$ is in $T^*_{\omega_1}$, whence $\aaqs \exists x(x\subseteq\omega\wedge\forall n
(n\in f_{\phi(s,x,\ba)}(s)\leftrightarrow n\in x))$ is in $T^*_{\xi}$ and therefore
 $\aaqs \exists x(x\subseteq\omega\wedge\forall n
(n\in f_{\phi(s,x,\sigma_\xi(\ba))}(s)\leftrightarrow n\in x))$ is true in $J'_{\omega\alpha}$,
where $\sigma_\xi$ is as in Lemma~\ref{69a}. 
So there is a club of sets $s$ such that $J'_{\omega\alpha}\models\exists x(x\subseteq\omega\wedge\forall n(n\in f_{\phi(s,x,\sigma_\xi(\ba))}(s)\leftrightarrow n\in x))$. 
Since $J'_{\omega\alpha}$ has only countably many reals (a consequence of Club Determinacy, see Theorem~\ref{tepois}), this club  is divided into countably many parts according to the $x\subseteq\omega$ such that $J'_{\omega\alpha}\models\forall n(n\in f_{\phi(s,x,\sigma_\xi(\ba))}(s)\leftrightarrow n\in x)$. One of those parts is stationary and therefore, by Club Determinacy,  contains a club. Hence 
$\exists x\aaqs (x\subseteq\omega\wedge\forall n(n\in f_{\phi(s,x,\ba)}(s)\leftrightarrow n\in x))$ is in $T_{\xi}^*$. Since $(M_\xi,\in,T_\xi,T_\xi^*)$ is an aa-mouse, there is $b\in M_\xi$ such that 
$\aaqs (b\subseteq\omega\wedge\forall n(n\in f_{\phi(s,x,\ba)}(s)\leftrightarrow n\in b))$ is in $T_{\xi}^*$. Hence $\aaqs(j_{\xi,\omega_1}(b)\subseteq\omega\wedge\forall n(n\in f_{\phi(s,x,\sigma_\xi(\ba))}(s)\leftrightarrow n\in j_{\xi,\omega_1}(b)))$ is true in $J'_{\omega\alpha}$, and therefore $r=j_{\xi,\omega_1}(b)$, 
a contradiction.%
\end{proof}

Let $\la\pi_\alpha:\alpha\le\omega_1,\alpha\ra$ be collapse functions such that  
\begin{enumerate}
\item $\pi_0=id:M_0=(J^T_\alpha,\in,T,T^*,(P)_0)=N_0=(J^{S_0}_{\zeta_0},\in,S_0,S^*_0,(\bP)_0)$

\item $\pi_{\alpha+1}:M_{\alpha+1}=\ult(M_\alpha)\cong N_{\alpha+1}=(J^{S_{\alpha+1}}_{\zeta_{\alpha+1}},\in,S_{\alpha+1},S_{\alpha+1}^*,(\bP^{\alpha+1})_{\alpha+1})$ 
\item $\pi_{\nu}: M_{\nu}\cong N_{\nu}=(J^{S_{\nu}}_{\zeta_{\nu}},\in,S_{{\nu}},S^*_{{\nu}},(\bP^{\nu})_{\nu}) \mbox{, limit $\nu$}.$

\end{enumerate}

Let $i_{\alpha,\beta}:N_\alpha\to N_\beta$ be defined by $i_{\alpha,\beta}(\pi_\alpha(a))=\pi_\beta(j_{\alpha,\beta}(a))$. We get the commuting diagram of Figure~\ref{iteration}. {
\def\do{\big\downarrow}
\def\ho{\hspace{12pt}}
\renewcommand{\arraystretch}{0.5}
{\scriptsize
\setlength{\arraycolsep}{2pt}
\begin{figure}
$$\begin{array}{cccccccccccc}
      &j_{01} &      &j_{12}&&& & &j_{\xi\xi+1}&&\\
\ho M_0&\longrightarrow     &\ho M_1&\longrightarrow  &\ho M_2&\ldots& \longrightarrow&\ho  M_\xi&\longrightarrow&\ho M_{\xi+1}&\ldots
&M_{\omega_1}\\
\\
\pi_0\do &        &\pi_1\do  &  &\pi_2\do && & \pi_\xi\do &&\hspace{-7mm}\pi_{\xi+1}\do &&\hspace{-10mm}\pi_{\omega_1}\do\\
\\
\ho N_0&\longrightarrow     &\ho N_1&\longrightarrow  &\ho N_2&\ldots& \longrightarrow&\ho  N_\xi&\longrightarrow&\ho  N_{\xi+1}&\ldots&N_{\omega_1}\\
      &i_{01} &      &i_{12}&&& & &i_{\xi\xi+1}&&&\\

\end{array}$$
\caption{The iteration.\label{iteration}}
\end{figure}
}

Suppose $\beta\in \On^{M_\gamma}$. Let $(J''_\beta)^{M_\gamma}$ the variant we obtain from $J'_\beta$ when we use $T_{\gamma}$ in place of $\tr$ in Definition~\ref{definJ}. 
 Recall that $M_\gamma$ is well-founded, so $\beta$ is well-founded but may not be a real ordinal. Respectively,  $(J''_\beta)^{N_\gamma}$.

\begin{lemma}\label{2ww2}Suppose $\beta\in N_{\omega_1}$.
\begin{enumerate} 
\item $\tr\rest\beta=\pi_{\omega_1}(T_{\omega_1})\rest \beta$.
\item $J'_\beta=(J''_\beta)^{N_{\omega_1}}$.
\end{enumerate}
\end{lemma}

\begin{proof}
Both claims are proved by simultaneous induction on $\beta$. Suppose the claims holds for $\beta=\pi_{\omega_1}(\bab)$, $\bab\in M_{\omega_1}$. Thus, $J'_\beta=(J''_\beta)^{N_{\omega_1}}$ and $\tr\rest\beta=\pi_{\omega_1}(T_{\omega_1})\rest\beta$.
By definition, 
$$
\begin{array}{lcl}
J'_{\beta+\omega}&=&\rud_{\tr}(J'_\beta\cup\{J'_\beta\})\\
(J''_{\beta+\omega})^{N_{\omega_1}}&=&\rud_{\pi_{\omega_1}(T_{\omega_1})}(J'_\beta\cup \{J'_\beta\}).\\
\end{array}
$$
We prove:
 $$\tr\rest \beta+\omega=\pi_{\omega_1}(T_{\omega_1})\rest \beta+\omega.$$
Suppose to this end, $(\beta,\phi(\ba))\in \pi_{\omega_1}(T_{\omega_1})$. Let $\bar{\beta}$ be such that $\beta=\pi_{\omega_1}(\bar{\beta})$  and $\bar{a}$ such that $\ba=\pi_{\omega_1}(\vec{\bar{a}})$.  Let $\gamma<\omega_1$ be such that  $\bar{\beta}=j_{\gamma\omega_1}(\beta^*)$, 
 and $\vec{\bar{a}}=j_{\gamma\omega_1}(\vec{a^*})$.  Thus $(\beta^*,\phi(\vec{a^*}))\in T_\gamma$. It follows that $\phi^{(J'_{\beta^*})}(\vec{a^*})\in T^*_{\gamma}$, whence
$\phi^{(J'_{\bar{\beta}})}(\vec{\bar{a}})\in T^*_{\omega_1}.$ By Theorem~\ref{rre}, $$(M_{\omega_1},E_{\omega_1},T_{\omega_1}, (P^{\omega_1})_{\omega_1})\models\phi^{(J'_{\bar{\beta}})}(\vec{\bar{a}})$$ and therefore
$(N_{\omega_1},\in,\pi_{\omega_1}(T_{\omega_1}))\models\phi^{(J'_\beta)}(\vec{a})$. Since $J'_\beta=(J''_\beta)^{N_{\omega_1}}$ and $\tr\rest\beta=\pi_{\omega_1}(T_{\omega_1})\rest\beta$, we obtain
$$(J'_\beta,\in,\tr\rest \beta)\models\phi(\vec{a})$$
i.e. $(\beta,\phi(\ba))\in\tr$. The other direction is similar. 
\end{proof}


We are now ready to prove the main result of this section. Since we assume Club Determinacy, there are only countably many reals in $\CAA$, but we show that there are, in the sense of $\CAA$, only $\aleph_1^{\CAA}$ many. Let $\omega_1^{aa}$ denote the $\omega_1$ of $\CAA$. The ordinal $\aleph_1^{aa}$ is in our case a countable ordinal in the sense of $V$.

{
\def\do{\big\downarrow}
\def\ho{\vspace{4pt}}
\setlength{\arraycolsep}{3pt}
\renewcommand{\arraystretch}{0.5}
\begin{figure}
$$\begin{array}{ccccc}
&&j_{\gamma\omega_1}&&\\
&(M_\gamma, E_\gamma, T_\gamma,T^*_\gamma)&\xrightarrow{\hspace*{1.5cm}}&(M_{\omega_1},E_{\omega_1},T_{\omega_1},T^*_{\omega_1})&\ho\\
&\rotatebox{90}{$\subset$}&&\rotatebox{90}{$\subset$}&\ho\\
&(J'_{{\beta^*}})^{M_\gamma}&\xrightarrow{\hspace*{1.5cm}}&(J'_{\bar{\beta}})^{M_{\omega_1}}&\ho\\
\\
&\pi_\gamma\do\hspace{7mm} &&\do\pi_{\omega_1}&\ho\\
\\
&({J'_{\beta}})^{N_\gamma}&\xrightarrow{\hspace*{1.5cm}}&({J'_{\beta}})^{N_{\omega_1}}&\\
&\rotatebox{270}{$\subset$}&&\rotatebox{270}{$\subset$}&\ho\\
&(N_\gamma,\in,\pi_\gamma(T_\gamma),\pi_\gamma(T^*_\gamma))&\xrightarrow{\hspace*{1.5cm}}&(N_{\omega_1},\in,\pi_{\omega_1}(T_{\omega_1}),\pi_{\omega_1}(T^*_{\omega_1}))&\ho\\
&&i_{\gamma\omega_1}&&\\
\end{array}$$
\caption{The levels.\label{levels}}
\end{figure}
}

\begin{theorem}\label{CHH}
 $\CH$ holds in $\CAA$.
\end{theorem} 

\begin{proof} 
Suppose $J'_{\alpha}$ is a stage where a new real $r$ of $\CAA$ is constructed, i.e.
\begin{equation}\label{raja}
r\in J'_{\alpha+\omega}\setminus J'_{\alpha},
\end{equation} and $\alpha$ is uniquely determined from $r$ by this equation. We show that $J'_{\alpha}\cap 2^\omega$ is countable in $\CAA$. It follows that $\CAA\cap 2^\omega$ has cardinality $\aleph_1$ in $\CAA$. Hence $\CAA\models \CH$.  

We can collapse $|\alpha|$ to $\aleph_1$ without changing $\CAA$ (Proposition~\ref{sigmaclosed}) or $\CAA\cap 2^\omega$. Also Club Determinacy is preserved in this forcing, because the forcing is countably closed. Therefore we can assume, w.l.o.g., that $|\alpha|=\aleph_1^V$. 

Let  $(M,\in,T,T^*)\in \CAA$ be {countable in $\CAA$} such that
\begin{equation}\label{prec}
\{r,\alpha,J'_\alpha,J'_{\alpha+\omega}\}\subseteq(M,{\in},T,T^*)\preccurlyeq (J'_{\aleph_2^V},\in,{\tr}\rest{\aleph_2^V},{\tr}_{\aleph_2^V}).
\end{equation}  
Let us use, as above,  $J''_{\beta}$ for $\beta\in M$ to denote $J'_{\beta}$ constructed using $T_{\omega_1}$ instead of $\tr$. 
So now by (\ref{prec}),
\begin{equation}\label{rajas}
r\in J''_{\alpha+\omega}\setminus J''_{\alpha}.
\end{equation}
The idea of the rest of the proof is the following. We iterate $(M,\in,T,T^*,(P)_0)$, $P=\emptyset$, inside $\CAA$ until we obtain $(M_{\omega_1},\in,T_{\omega_1},T_{\omega_1}^*,(P^{\omega_1})_{\omega_1})$. 
We have shown in Lemma~\ref{2ww2} that  $J''_{\alpha}=J'_{\alpha},$ whence $J'_{\alpha}\cap 2^\omega\subseteq M_{\omega_1}.$ Lemma~\ref{nonew} implies 
$M_{\omega_1}\cap 2^\omega\subseteq M.$ It will follow that 
$J'_{\alpha}\cap 2^\omega$ is countable, as we wished to demonstrate.
By Lemma~\ref{nonew},  no new reals are generated in the  iteration.
By Lemma~\ref{2ww2}
$J'_{\beta}=(J''_\beta)^{N_{\omega_1}}$.
Now:
{\renewcommand{\arraystretch}{1.5}
$$\begin{array}{lcl}
r&=&i_{0\omega_1}(\pi_0(r))\\
&\in& i_{0\omega_1}(\pi_0(J''_{\alpha+\omega}))\setminus i_{0\omega_1}(\pi_0(J''_{\alpha}))\\
&=&
(J'_{i_{0\omega_1}(\pi_0(\alpha))+\omega})^{N_{\omega_1}}\setminus (J'_{i_{0\omega_1}(\pi_0(\alpha))})^{N_{\omega_1}}\\
&=&
J'_{i_{0\omega_1}(\pi_0(\alpha))+\omega}\setminus J'_{i_{0\omega_1}(\pi_0(\alpha))}.
\end{array}$$} %
By (\ref{raja}), $i_{0\omega_1}(\pi_0(\alpha))=\alpha\in N_{\omega_1}$ and further by Lemma~\ref{2ww2}, $i_{0\omega_1}(\pi_0(J'_\alpha))=J'_\alpha$.
Thus all the reals of $J'_\alpha$ are in $N_{\omega_1}$ and hence in $M$, and therefore they are countably many only.
\end{proof}

The above proof shows that $\CAA\models 2^{\aleph_\alpha}=\aleph_{\alpha+1}$ for all $\alpha\le \omega_1 (=\omega_1^V)$. For $\alpha<\omega_1$ the above proof works, and for $\alpha=\omega_1$ the claim therefore follows from the fact (Theorem~\ref{tepois}) that $\omega_1$ is  measurable in $\CAA$\footnote{Work 
in progress by a SQuaRE group  shows that if Club Determinacy holds, $\CAA$ satisfies full ${\GCH}$ and  has no inner model with a Woodin cardinal.}\setcounter{footnoteValueSaver}{\value{footnote}}.

\begin{theorem}\label{delta}
There is a  $\Delta^1_3$ well-ordering of the reals in $\CAA$.
\end{theorem}

\begin{proof} 
We show that the canonical well-order $\prec$ of $\CAA$ is $\Delta^1_3$. The proof of Theorem~\ref{CHH} essentially shows that for any reals $x,y$ in $\CAA$:
$$x\prec y\iff \exists z\subseteq \omega(\mbox{ $z$ codes an  aa-mouse $M$ such that }$$
$$x,y\in M\mbox{ and }M\models ``x\prec y").$$
Being a real that codes a countable aa-mouse is $\Pi^1_2$. Hence the right hand side of the equivalence is $\Sigma^1_3$ and the claim follows.\end{proof}

\begin{corollary}\label{nowoodin}
There are no Woodin cardinals in $\CAA$\footnotemark[\value{footnoteValueSaver}].
\end{corollary}

\begin{proof}
The proof of Theorem~\ref{delta}  shows that, assuming Club Determinacy, there is a $\Delta^1_3$-well-ordering of the reals, this well-ordering is in $\CAA$ and $\Delta^1_3$ in $\CAA$. Suppose there is a Woodin cardinal in $\CAA$. There would be a measurable cardinal above it by Theorem~\ref{tepois}. A measurable cardinal above a Woodin cardinal implies $\Sigma^1_2$-determinacy (\cite{MR959109}). On the other hand, $\Sigma^1_2$-determinacy implies that $\Sigma^1_3$-sets of reals are Lebesgue measurable which contradicts the existence of a $\Sigma^1_3$-well-ordering of the reals.   
\end{proof}

\begin{theorem}
$\bDiamond$ holds in $\CAA$.
\end{theorem}

\begin{proof}
We define $S_\alpha$ for $\alpha<\omega_1^{aa}$ as follows: Let $(C,X)$ be the $\prec_{\CAA}$-minimal pair $(C,X)$, where $C\subseteq\alpha$ is a club and $X\cap \beta\ne S_\beta$ for $\beta\in C$.  We then let $S_\alpha=X$. Suppose the set $\mathcal{S}=\langle S_\alpha:\alpha<\omega_1^{aa}\rangle$ thus built is not a $\bDiamond$-sequence in $\CAA$. Then there are $X\subseteq\omega_1^{aa}$  and a club $C\in \CAA$   such that $C\subseteq\omega_1^{aa}$ and $\beta\in C$ implies $X\cap\beta\ne S_\beta$. Let $\delta$ be minimal such that such a pair can be found in $J'_\delta$. W.l.o.g., $\delta<\aleph_2^V$. 
Let  $(M,T)\in \CAA$ be countable (in $\CAA$) such that 
\begin{equation}\label{precc}
\{\mathcal{S},\delta,J'_\delta,{\tr}\rest{\delta},{\tr}_{\delta},\omega_1^{aa},C,X\}\subseteq(M,\in,T,T^*)\preceq (J'_{\aleph_2^V},\in,{\tr}\rest{\aleph_2},{\tr}_{\aleph_2}).
\end{equation}  
 We build models $M_\xi$ and $N_\xi$ as well as elementary mappings $i_{\alpha\beta},j_{\alpha\beta}$ and isomorphisms $\pi_\alpha$ for $\alpha<\beta\in N_{\omega_1}$ with $M_0=M$ as in the proof of Theorem~\ref{CHH}, see Figure~\ref{iteration}. Let $\alpha=M\cap\omega_1^{aa}$, $\bar{C}=C\cap\alpha$ and $\bar{X}=X\cap\alpha$. Clearly $\alpha\in C$, as $C$ is club. Let $\delta^*=i_{0\omega_1}(\pi_0(\delta))$. The ordinal $\delta^*$ is the minimal $\delta^*$ such that there is a counterexample such as $(\bar{C},\bar{X})$ in $J'_{\delta^*}$ in $N_{\omega_1}$. The ordinal $\alpha$ is below the critical point of $i_{01}$, whence $$i_{0\omega_1}(\langle S_\beta:\beta<\alpha\rangle)=\langle S_\beta:\beta<\alpha\rangle.$$ Therefore, according to our definition, $S_\alpha=\bar{X}$, contradicting $\alpha\in C$.

\end{proof}

\section{Variants of stationary logic}

There are several variants of stationary logic. The earliest variant is based on the following quantifier   introduced in  \cite{MR0376334}, a predecessor of the quantifier $\aaq$:


\begin{definition}
$\mm\models Q^{St} xyz\phi(x,{\ba})\psi(y,z,{\ba})$ if and only if $(M_0,R_0)$, where
$M_0=\{b\in M : \mm\models\phi(b,{\ba})\}$ and 
$R_0=\{(b,c)\in M : \mm\models\psi(b,c,{\ba})\}$ is an $\aleph_1$-like linear order and the set $\I$ of initial segments of $(M_0,R_0)$  with an $R_0$-supremum in $M_0$ is stationary  in the set $\DD$  of all (countable) initial segments of $M_0$  in the following sense: If $\J\subseteq\DD$ is unbounded in $\DD$ (i.e. $\forall x\in\DD\exists y\in\J(x\subseteq y)$) and $\sigma$-closed in $\DD$ (i.e. if $x_0\subseteq x_1\subseteq\ldots $ in $\J$, then $\bigcup_nx_n\in\J$), then $\J\cap\I\ne\emptyset$.\end{definition}
\noindent The logic $\LL(Q^{St})$, a sublogic of $\LL(\aaq)$, is recursively axiomatizable and $\aleph_0$-compact \cite{MR0376334}. We call this logic {\em Shelah's stationary logic}, and denote $C(\LL(Q^{St}))$ by $C(\aam)$\footnote{It should be noted that there is no difference between $C(\aam)$ and $C_{\old}(\aam)$.}. For example, we can say in the logic $\LL(Q^{St})$ that a formula $\phi(x)$ defines a stationary (in $V$) subset of $\omega_1$ in a transitive model $M$ containing $\omega_1$ as an element as follows:
 
$$M\models \forall x(\phi(x)\to x\in\omega_1)\wedge Q^{St} xyz\phi(x)(\phi(y)\wedge\phi(z)\wedge y\in z).$$
Hence
$$C(\aam)\cap \F_{\omega_1}\in C(\aam),$$
where $\F_{\omega_1}$ is the club-filter on $\omega_1$,
and in fact the set $D=C(\aam)\cap \F_{\omega_1}$ suffices to characterise $C(\aam)$ completely:
$C(\aam)=L[D]$, as we shall prove in the next Lemma. In particular, $C(\aam)\subseteq \CAA$.

\begin{lemma}\label{mmqq}
$C(\aam)=L[D]$. 
\end{lemma}

\begin{proof}
We already know $L[D]\subseteq C(\aam)$. We prove the converse by induction on the structure of 
$C(\aam)$. This boils down to showing that we can recognize in $L[D]$ whether a subset $M_0$  of an $\aleph_1$-like linear order $R_0$, both $M_0$ and $R_0$  in $L[D]$, satisfy $Q^{St}$ in the sense that the set of initial segments of $R_0$ with supremum in $M_0$ is stationary in the set of all initial segments of $R_0$. The model   $L[D]$ knows a cofinal mapping $\pi$ from an ordinal $\alpha$ into the domain of $R_0$. Since $R_0$ is $\aleph_1$-like, $\alpha=\omega_1^V$. Now $L[D]$ can use $\pi$ and $D$ to decide whether $M_0$ and $R_0$ satisfy $Q^{St}$. \end{proof}

\begin{theorem}\label{werlkj}
 If there are two Woodin cardinals, then $D=C(\aam)\cap \F_{\omega_1}$ is an ultrafilter in $C(\aam)$. In particular, $C(\aam)\models \GCH$.
 \end{theorem}

\begin{proof}
We know $C(\aam)=L[\F_{\omega_1}]$. 
We show that $D=\F_{\omega_1}\cap C(\aam)$ measures every set in $C(\aam)$. Let us  assume the contrary. We take a minimal $\alpha$ such that there is a (minimal) set $B\subseteq\omega_1$ in $J'_\alpha$ (the hierarchy generating $C(\aam)$) such that $B\notin D$ and $\omega^V_1\setminus B\notin D$. The logic $\LL(\aam)$ satisfies a Downward L\"owenheim Skolem Tarski Theorem down to $\aleph_1$ (\cite{MR0376334}). Hence $|\alpha|\le\aleph_1$.
As in the beginning of the proof of Theorem~\ref{main1}, we can assume, w.l.o.g., that $\delta^1_2=\omega_2$ and we have still one Woodin cardinal $\delta$ left.
Let $G$ be $Q_{<\delta}$-generic  and $j:V\to M\subseteq V[G]$ the generic ultrapower embedding. Let $j(\alpha)=\beta$. Now $j(B)$ is a stationary co-stationary subset of $\delta$
($=\omega_1^M$) in $M$. Moreover, $\beta$ is the minimal ordinal for which there is such a set in $J'_{\beta}$ in $M$, and  $j(B)$  is  minimal   such a set in $\LL(\aam)$. As in the proof of Theorem~\ref{main1} we can now argue that $j(B)\in V$. We get a contradiction by taking two different generic sets for $Q_{<\delta}$, one containing $B$ and the other containing $\omega_1\setminus B$.\end{proof}

\begin{proposition}\label{439829003}
If $0^\#$ exists, then $0^\#\in C(\aam)$.
\end{proposition}

\begin{proof}
  Assume $0^\sharp$.
A first order formula $\phi(x_1,\ldots,x_n)$ holds in $L$ for an increasing sequence of indiscernibles below $\omega_1^V$ if and only if there is a club $C$ of ordinals $<\omega_1^V$ such that every increasing sequence $a_1<\ldots <a_n$ from $C$ satisfies  $\phi(a_1,\ldots,a_n)$ in $L$. Similarly,
$\phi(x_1,\ldots,x_n)$ does not hold in $L$ for an increasing sequence of indiscernibles below $\omega_1^V$ if and only if there is a club of ordinals $a_1<\omega_1$ such that there ia a club of ordinals $a_2$ with $a_1<a_2<\omega_1$ such that $\ldots$ such that  there is a  club of ordinals $a_n$ with $a_{n-1} <a_n<\omega_1$ satisfying $\neg\phi(a_1,\ldots,a_n)$. From this it follows that $0^\#\in C(\aam)$. 
  \end{proof}

\begin{theorem}
It is consistent relative to the consistency of ${\zfc}$ that
$$C^*\nsubseteq C(\aam)\wedge C(\aam)\nsubseteq C^*.$$
\end{theorem}

\begin{proof}
 We force over $L$ and first we add two Cohen reals $r_0$ and $r_1$, to obtain $V_1$. Now we use  modified Namba forcing to make $\cof(\aleph^L_{n+1})=\omega$ if and only if $n\in r_0$.  This forcing satisfies the $S$-condition (see \cite{kmv}), and therefore will not---by \cite{MR924672}---kill  the stationarity of any stationary subset of $\omega_1$. The argument is essentially the same as for Namba forcing. Let the extension of $V_1$ by $\oP$ be $V_2$. In $V_2$ we have $C(\aam)=L$ because we have not changed stationary subsets of $\omega_1$. But $V_2\models r_0\in C^*$.

 Let $S_n$, $n<\omega$, be in $L$ a definable sequence of disjoint stationary subsets of $\omega_1$ such that $\bigcup_n S_n=\omega_1$. Working in $V_2$, we use the canonical forcing notion  which kills the stationarity of $S_n$ if and only if $n\in r_1$. Let the resulting model be $V_3$. The cofinalities of  ordinals are the same in $V_2$ and $V_3$, whence  $(C^*)^{V_2}$ is the same as $(C^*)^{V_3}$. Thus $V_3\models r_1\in C(\aam)\setminus C^*.$ Now we argue that $V_3\models C(\aam)=L(r_1)$. First of all, $L(r_1)\subseteq C(\aam)$ by the construction of $V_3$. Next we prove by induction on the construction of $C(\aam)$ as a hierarchy $J'_\a$, $\a\in On$, that  $J'_\a\subseteq L(r_1)$. When we consider $J'_{\a+1}$ and assume $J'_\a\subseteq L(r_1)$, we have to decide whether a subset $S$ of $\omega_1$, constructible from $r_1$, is stationary or not. The set $S$ is stationary in $V_3$ if and only if it is stationary in $L(r_1)$ and it is not included modulo the club filter in $\bigcup_{n\in r_1}S_n$. Thus $J'_{\a+1}\subseteq L(r_1)$. 
 
In $V_3$ the real $r_0$ is in $C^*\setminus C(\aam)$ and the real $r_1$ is in $C(\aam)\setminus C^*$.
\end{proof}

The logics $\LL(Q^{\cf}_\omega)$, giving rise to $C^*$, and $\LL(\aam)$, giving rise to $C(\aam)$, are two important logics, both introduced by  Shelah. Since $\LL(Q^{\cf}_\omega)$ is {\em fully} compact,  $\LL(\aam)$ cannot be a sub-logic of it. On the other hand, it is well-known and easy to show that $\LL(Q^{\cf}_\omega)$ {\em is} a sub-logic of $\LL(\aaq)$. Therefore it is interesting to note the following corollary to the above theorem:

\begin{corollary}
 It is consistent, relative to the consistency of {\zfc},  that $\LL(Q^{\cf}_\omega)\nsubseteq \LL(Q^{St})$ and hence $\LL(Q^{St})\ne \LL(\aaq)$.
\end{corollary}

We do not know whether $\LL(Q^{\cf}_\omega)\subseteq \LL(\aam)$ or  $\LL(\aam)=\LL(\aaq)$ is consistent.

A modification of $C(\aam)$ is the following $C(\aal)$:

\begin{definition}
$\mm\models Q^{St,0} xyzu\phi(x,y,{\ba})\psi(u,{\ba})$ if and only if 
$M_0=\{(b,c)\in M : \mm\models\phi(b,c,{\ba})\}$ is a linear order of cofinality $\omega_1$ and every club of initial segments has an element with supremum in 
$R_0=\{b\in M : \mm\models\psi(b,{\ba})\}$. The inner model $C(\aal)$ is defined as $C(\LL(Q^{St,0})).$\end{definition}

\begin{proposition} If there are two Woodin  cardinals, then
 $C(\aal)\models``\aleph^V_1$ is a measurable cardinal".
\end{proposition}

\begin{proof}
 The proof of this is---{\em mutatis mutandis}---as  the proof for $C(\aam)$.

\end{proof}

\begin{proposition}\label{dagger}
 If $0^\dagger$ exists, then $0^\dagger\in C(\aal)$.
 \end{proposition}

\begin{proof}Assume $0^\dagger$. There is a club class of indiscernibles for the inner model $L[U]$ where $U$ is (in $L[U]$) a normal measure on an ordinal $\delta$. Let us choose an indiscernible $\a$ above $\delta$ of $V$-cofinality $\omega_1^V$. We can define $0^\dagger$ as follows: An increasing sequence of indiscernibles satisfies a given formula $\phi(x_1,\ldots,x_n)$ if and only if there is a club $C$ of ordinals below $\alpha$  such that every increasing sequence $a_1<\ldots <a_n$ from $C$ satisfies  $\phi(a_1,\ldots,a_n)$ in $L[U]$. Similarly,
$\phi(x_1,\ldots,x_n)$ does not hold in $L[U]$ for an increasing sequence of indiscernibles below $\a$ if and only if there is a club of ordinals $a_1<\a$ such that there is a club of ordinals $a_2$ with $a_1<a_2<\a$ such that $\ldots$ such that  there is a  club of ordinals $a_n$ with $a_{n-1} <a_n<\a$ satisfying $\neg\phi(a_1,\ldots,a_n)$. From this it follows that $0^\dagger\in C(\aal)$.

\end{proof}

\begin{corollary}
If there are two Woodin cardinals, then $C(\aam)\ne C(\aal)$. Then also the logics  
$\LL(Q^{St})$ and $\LL(Q^{St,0})$ are non-equivalent.\end{corollary}

\begin{proof}
If there are two Woodin cardinals, then then $0^\dagger$ exists and $C(\aam)$ does not contain $0^\dagger$ by Theorem~\ref{werlkj}, while $C(\aal)$ does contain by  
Proposition~\ref{dagger}.
\end{proof}

Note that it is probably possible to prove the non-equivalence of the logics $\LL(Q^{St})$ and $\LL(Q^{St,0})$  in ZFC with a model theoretic argument using the exact definitions of the logics and by choosing the structures very carefully. But the non-equivalence result given by the above Corollary is quite robust in the sense that it is not at all sensitive to the exact definitions of the logics as long as the central separating feature, manifested in structures of the form $(\a,<)$, is respected.

\section{Open problems}

There are many open questions about $\CAA$. We list here the some of the most prominent ones:

\begin{enumerate}
\item Do regular cardinals of $V$ have stronger large cardinal properties in $\CAA$ than measurability,  under the assumption of large cardinals?
\item Which inner models for large cardinals (below one Woodin cardinal) exist inside $\CAA$?
\item  What is the consistency strength of Club Determinacy? 
\item Is AC true in $C_{\old}(\aaq)$? Is $C_{\old}(\aaq)=\CAA$?
\end{enumerate}

\section{Appendix: A counter-example to AC in $C_o(L^*)$}

Consider the quantifier
$$\begin{array}{l}
M\models Q_n^{ST}xyz\phi(x,\ba)\psi(y,z,\ba)\iff\psi(\cdot,\cdot,\ba)\mbox{ has  order-type $\aleph_{n+1}$}\\
\qquad\mbox{ and }\phi(\cdot,\ba)\mbox{ is a stationary set of points }
\mbox{of cofinality } \aleph_n \mbox{ in } \psi(\cdot,\cdot,\ba). 
\end{array}$$
We let $L^*$ be the extension of first order logic by the infinitely many quantifiers $Q_n^{ST}$, $n<\omega$.
Note that $C_o(L^*)\models ZF$.

\begin{proposition}
$Con(ZF)$ implies $Con(C_o(L^*)\models\neg AC)$.
\end{proposition}

\begin{proof}
We start with $V=L$. Let $S_{n,m}$, $m<\omega$, be disjoint stationary sets of ordinals $<\aleph_{n+1}$ of cofinality $\aleph_n$ such that the set  $\{\alpha<\aleph_{n+1}:\cof(\alpha)=\aleph_n\}\setminus \bigcup_{m}S_{n,m}$ is also stationary. We force mutually generic Cohen-reals $a_n$, $n<\omega$. Let us call the po-set of this forcing $\oQ$.
Let $\oP_n(a_n)$ force an $\aleph_n$-closed unbounded set $C_{n+1}$ into the set $\bigcup_{m\in a_n}S_{n,m}$. Let $\ba=\langle a_n:n<\omega\rangle$ and let $\oP(\ba)$ be the product of $\oP_n(a_n)$, $n<\omega$. In the extension by $\oP(\ba)$ we have for all $m,n<\omega$: $$m\in a_n\iff S_{n,m}\mbox{ is stationary},$$ whence  $\ba\subseteq C_o(L^*)$. 

Assume now $V=L[\ba][\langle C_{n+1}:n<\omega\rangle]$.
\medskip

\noindent{\bf Claim A:} $C_o(L^*)\subseteq L[\ba]$ and $\langle L'_\alpha:\alpha<\delta\rangle\in L[\ba]$ for all $\delta$.
\medskip

\begin{proof}
For a proof by induction, suppose $L'_\alpha\in L[\ba]$. Suppose $Z\in L'_{\alpha+1}$ is an $L^*$-definable (with parameters) subset of $L'_\alpha$. We shall show that $Z$ is in $L[\ba]$. Since we proceed by induction, this boils down to showing that if  $X\in L[\ba]$ is set of ordinals $<\aleph_{n+1}\le\alpha$ of cofinality $\aleph_n$, then we can decide in  
$L[\ba]$ whether $X$ is stationary in $V$ or not.
To this end we shall show that the following conditions are equivalent:
\begin{description}
\item [($*$)] $X$ is stationary in $V$.
\item [($**$)] There are $m\in a_n$ and $Y\subseteq X$ such that $Y\in L$ and $Y\cap S_{n,m}$ is stationary in $L$. 
\end{description}

\noindent $(*)\to (**)$: Since $L[\ba]$ is obtained from $L$ by the countable po-set $\oQ$, the $V$-stationary set $X$ is a countable union of sets in $L$. Thus $X$  contains a $V$-stationary subset $Y$ in $L$. We have forced an $\aleph_n$-closed unbounded set $C_{n+1}$ into the set $\bigcup_{m\in a_n}S_{n,m}$. There must be $m\in a_n$ such that $Y\cap S_{m,n}\in L$ is stationary in $V$, hence in $L$. 

\noindent $(**)\to (*)$: Suppose $m$ and $Y$ are as in $(**)$. Thus $Y\cap S_{n,m}$ is stationary in $L$. Adding the Cohen reals preserves the stationarity of $Y\cap S_{n,m}$.  Thus $Y\cap S_{n,m}$ is stationary in $L[\ba]$. If $k< n$, the po-set $\oP_k(a_k)$ is of cardinality $<\aleph_{k+1}$. Hence it does not kill the stationarity of $Y\cap S_{n,m}$. If $k>n$, the po-set $\oP_k(a_k)$ is  $<\aleph_{n+1}$-distributive. Hence it does not kill the stationarity of $Y\cap S_{n,m}$. Finally, $\oP_n(a_n)$ forces the  $\aleph_n$-closed unbounded set $C_{n+1}$ into the set $\bigcup_{l\in a_n}S_{n,l}$. But $Y\cap S_{n,m}\subseteq \bigcup_{l\in a_n}S_{n,l}.$ It is a standard fact about the club shooting forcing that if you add a generic club through the complement of a stationary set $S$ then the stationarity of any stationary set disjoint from $S$ is preserved. Hence the stationarity of $Y\cap S_{n,m}$ is preserved by $\oP_n(a_n)$. All in all, $X$ is stationary in $V$. We have proved the equivalence of ($*$) and $(**)$ and thereby Claim A.

\end{proof}

Let us fix $n^*<\omega$ and form a new sequence
$$\ba^*=\langle a_l:l<n^*\rangle^\frown\langle a^*_{n^*}\rangle^\frown \langle a_l:l>n^*\rangle,$$
where $a^*_{n^*}$ is a finite modification of $a_{n^*}$. Obviously, $L[\ba]=L[\ba^*]$. Let $M_1$ be obtained from 
$L[\ba]$ by forcing with the po-set $\oP(\ba)$ and $M_2$ from $L[\ba^*]$ (i.e. $L[\ba]$) by forcing with $\oP(\ba^*)$.
\medskip

\noindent{\bf Claim B:} $(C_o(L^*))^{M_1}=(C_o(L^*))^{M_2}$.
\medskip

\begin{proof}
We prove by induction on $\alpha$ that $(L'_\alpha)^{M_1}=(L'_\alpha)^{M_2}$ and
$(L'_\alpha)^{M_1}, (L'_\alpha)^{M_2}\subseteq L[\ba]$. Suppose this holds for $\alpha$.  Suppose $Z\in (L'_{\alpha+1})^{M_1}$ is in the sense of $M_1$ an $L^*$-definable (with parameters) subset of $(L'_\alpha)^{M_1}$. We shall show that $Z$ is in $(L'_{\alpha+1})^{M_2}$. Since we proceed by induction, we have to show that if  $X\subseteq(L'_\alpha)^{M_1}, X\in L[\ba]$ is a set of ordinals $<\aleph_{n+1}\le\alpha$ of cofinality $\aleph_n$, then $M_2$  can detect whether  $X$ is stationary in $M_1$ or not, and $M_1$  can detect whether  $X$ is stationary in $M_2$ or not.  By the equivalence of $(*)$ and $(**)$ above,  this boils down to detecting whether there is  $m\in a_n$  and $Y\subseteq X$, $Y\in L$, such that $Y\cap S_{m,n}$ is stationary in $L$, and respectively in $M_2$, switching $\ba$ to $\ba^*$. The question is non-trivial only if $n=n^*$, which we now assume. There is only a finite difference between  $a_n$ and $a^*_n$ and the criterion ``$Y\cap S_{m,n}$ is stationary in $L$" yields the same answer in $M_1$ and $M_2$. Therefore $M_1$ can detect whether $X$ is stationary in $M_2$, or not and vice versa.
\end{proof}

We are now ready to prove that AC fails in $C_o(L^*)$ in the model $M_1$. Suppose $\phi(x,y,\vec{b})$ defines, with parameters $\vec{b}$, a well-order $\prec$ of the reals of $M_1$. Let $n^*$ be large enough to be greater than any $m$ such that  $Q_m^{ST}$ occurs in $\phi(x,y,\vec{b})$ or in the definitions of the parameters $\vec{b}$, computed recursively. 
Let $p\in\oQ$ force that $a_{n^*}$ is the $\alpha$th real in the well-order $\prec$. Let us modify $a_{n^*}$ to $a_{n^*}^*$ so that they still agree about integers in the domain of $p$. We obtain $\ba^*$ and $M_2$, as above.  Now $M_1\models "a_{n^*}$ is the $\alpha$th real of $C_{o}(L^{*})$" and $M_2\models "a_{n^*}^*$ is the $\alpha$th real of $C_{o}(L^{*})$". However, $M_1$ and $M_2$ agree about $\prec$, because of the way we have chosen $n^*$, a contradiction.
\end{proof}

\bibliographystyle{plain}
\bibliography{KMVbib}

\bigskip

Juliette Kennedy

Department of Mathematics and Statistics

University of Helsinki
\medskip

Menachem Magidor

Department of Mathematics

Hebrew University Jerusalem
\medskip

Jouko V\"a\"an\"anen 

Department of Mathematics and Statistics

University of Helsinki

and

Institute for Logic, Language and Computation

University of Amsterdam

\end{document}